\documentclass[a4paper]{article}
\usepackage{authblk}
\usepackage[a4paper]{geometry}
\usepackage{lmodern}
\usepackage{amsfonts}
\usepackage{amssymb}
\usepackage{mathrsfs}

\usepackage{amsrefs}

\usepackage{interval}

\usepackage{mathtools}

\usepackage[inline]{enumitem}
\setlist[enumerate,1]{label={\textup(\arabic*)}}

\usepackage{amsthm}
\theoremstyle{definition}
\newtheorem{defi}{Definition}[section]
\theoremstyle{plain}
\newtheorem{exam}[defi]{Example}
\newtheorem{lemm}[defi]{Lemma}
\newtheorem{prop}[defi]{Proposition}
\newtheorem{coro}[defi]{Corollary}
\newtheorem{theo}[defi]{Theorem}
\newtheorem{nota}[defi]{Notation}
\theoremstyle{remark}
\newtheorem{rema}[defi]{Remark}

\newcommand{\leadmathskip}{\hspace{1.2em}}

\usepackage{hyperref}
\usepackage{xcolor}
\hypersetup{
  colorlinks,
  linkcolor={brown!50!black},
  citecolor={gray},
  urlcolor={blue},
}

\numberwithin{equation}{section}

\newcommand{\RD}{{\rm (RD)} }

\DeclareMathOperator{\adj}{Ad}
\DeclareMathOperator{\alg}{alg}
\DeclareMathOperator{\aut}{Aut}

\DeclareMathOperator{\id}{id}
\DeclareMathOperator{\image}{Image}

\DeclareMathOperator{\inn}{Inn}
\DeclareMathOperator{\irr}{Irr}
\DeclareMathOperator{\morph}{Mor}
\DeclareMathOperator{\orbit}{Orb}
\DeclareMathOperator{\out}{Out}

\DeclareMathOperator{\pol}{Pol}

\DeclareMathOperator{\psl}{PSL}

\DeclareMathOperator{\speciallinear}{SL}

\DeclareMathOperator{\tr}{Tr}
\DeclareMathOperator{\vect}{Vect}

\DeclarePairedDelimiter{\norm}{\|}{\|}
\DeclarePairedDelimiter{\abs}{\lvert}{\rvert}
\DeclarePairedDelimiterX{\pairing}[2]{\langle}{\rangle}%
{\mskip1mu#1\mskip2mu,\mskip3mu#2\,}
\DeclarePairedDelimiterX{\pair}[2]{\lparen}{\rparen}%
{\mskip1mu#1\mskip2mu,\mskip3mu#2\,}
\DeclarePairedDelimiterX\set[1]\lbrace\rbrace{\def\given{\;\delimsize\vert\;}#1}

\title{Some Examples of Bicrossed Products with the Rapid Decay Property}
\author{Hua Wang}
\date{}
\affil{Institute for Advanced Study in Mathematics, Harbin Institute of Technology \\
  Harbin 15001, China \\
  Email: \href{mailto:huawang@phare.normalesup.org}{huawang@phare.normalesup.org} \\
}

\AtBeginDocument{%
  \def\MR#1{}
}

\begin{document}

\maketitle

\begin{abstract}
  We consider bicrossed products obtained by twisting compact semi-direct
  products by a suitable finite subgroup. We give a practical criterion for the
  rapid decay property and polynomial growth of the dual of such bicrossed
  products under a mild restriction. Using this theory, we construct concrete
  new examples of discrete quantum groups possessing the rapid decay property
  but not growing polynomially.
\end{abstract}

\section{Introduction}
\label{sec:e8a855b65ee329aa}

The rapid decay property, aka. Property \RD, be it for discrete groups
(Jolissaint~\cite{MR943303}) or discrete quantum groups
(Vergnioux~\cite{MR2329000}), is an interesting approximation property due to
its connection with \( K \)-theory, \( L^{2} \)-homology et cetera. From its
introduction by Kac (\cite{MR0229061}) in the 1960s to the culminating work of
Vaes \& Vainerman (\cite{MR1970242}), bicrossed products, as a large class of
non-commutative and non-cocommutative quantum groups, play an important role
in the development of the still rapidly growing of the study of quantum
groups. In \cite{MR4345210}, we give a systematic theoretical study of property
\RD of (the dual of) bicrossed products \( \Gamma \bowtie G \) of a matched pair
\( (\Gamma, G) \) of a compact group with a discrete one. The main results there
being that \( \widehat{\Gamma \bowtie G} \) has \RD if and only if both
\( \Gamma \) and \( \widehat{G} \) have \RD in a compatible manner
(Theorem~\ref{theo:66c73d7c75f29067} below has a more precise formulation). One
of the key component in formulating this result is the notion of matched pair of
length functions, which is closely related to the fusion structure of
irreducible representations of the compact quantum group \( \Gamma \bowtie G \).
In practice, it is quite difficult to verify that the witnessing length
functions for the relevant approximation properties like \RD or polynomial
growth are matched. This is the main reason of the lack of concrete examples in
the theoretical study of \cite{MR4345210}, which is a major drawback of the
study there.

This paper picks up where \cite{MR4345210} left off, namely, the construction of
concrete examples supporting the theory developed there. We work in a particular
class of bicrossed products already considered in \cite{MR3743231}, which we now
briefly describe and refer to \S~\ref{sec:e9d9394edcf4dfb6} for more details. We
start from a nontrivial semidirect product \( G \rtimes \Gamma \) with \( G \) a
compact group and \( \Gamma \) a discrete one, as well as a finite subgroup
\( \Lambda \) of \( \Gamma \) that does not lie entirely in the center of
\( \Gamma \). One can then show that the pair \( (\Gamma, G \rtimes \Lambda) \)
is matched in a natural way (again see \S~\ref{sec:e9d9394edcf4dfb6}), where the
underlying left action
\( \alpha^{\Lambda} : \Gamma \curvearrowright G \rtimes \Lambda \) and right
action \( \beta^{\Lambda} : \Gamma \curvearrowleft G \rtimes \Lambda \) are both
nontrivial, hence yields a non-commutative, non-cocommutative bicrossed product
\( \Gamma \bowtie (G \rtimes \Lambda) \). We restrict our attention to the study
of \RD and the closely related notion of polynomial growth of the dual
\( \widehat{\Gamma \bowtie (G \rtimes \Lambda)} \). This has the advantage that
the underlying representation theory can be divided to two well-understood part:
the first part is the representation theory of the bicrossed products studied in
\cite{MR4345210}, which, modulo the relations determined by the matched pair
\( (\Gamma, G \rtimes \Lambda) \), reduces to the second part, namely the
representation theory of the semidirect product \( G \rtimes \Lambda \), which
is a special case studied in \cite{wang2019representations}. Moreover, the
finiteness of the group \( \Lambda \) provides a way to control the length
functions, which eventually leads to a characterization result that gets rid of
the technically involved notion of matched pair of length functions (see
Definition~\ref{defi:46e581e887d239e1}), this is the contents of
Theorem~\ref{theo:54211248e165a397}.

We can now state the main theoretical results of this paper consisting of the
following two theorems.

\begin{theo}
  \label{theo:69214233f77b0115}
  In the above settings. If there is a \( \Gamma \)-invariant length function
  \( l_{\widehat{G}} \) on \( \widehat{G} \), and a
  \( \beta^{\Lambda} \)-invariant length function \( l_{\Gamma} \) on
  \( \Gamma \), such that both \( \pair*{\widehat{G}}{l_{\widehat{G}}} \) and
  \( \pair*{\Gamma}{l_{\Gamma}} \) have polynomial growth (resp.\ \( (RD) \)),
  then the dual of the bicrossed product, namely
  \( \widehat{\Gamma \bowtie (G \rtimes \Lambda)} \) also has polynomial growth
  (resp.\ \( (RD) \)).
\end{theo}

\begin{theo}
  \label{theo:54211248e165a397}
  Let \( \widetilde{\tau} : \Gamma \to \out(G) \) be the composition of
  \( \tau : \Gamma \to \aut(G) \) with the canonical projection
  \( \aut(G) \to \out(G) \). If \( \image(\widetilde{\tau}) \) is finite, then
  the following are equivalent:
  \begin{enumerate}
  \item \label{item:c24cd688ec1da3ff}
    \( \widehat{\Gamma \bowtie (G \rtimes \Lambda)} \) has polynomial growth
    (resp.\ \( (RD) \));
  \item \label{item:ac508fcd0e5a1c73} both \( \Gamma \) and \( \widehat{G} \)
    have polynomial growth (resp.\ \( (RD) \)).
  \end{enumerate}
\end{theo}

Theorem~\ref{theo:69214233f77b0115} provides a sufficient condition for
\( \widehat{\Gamma \bowtie (G \rtimes \Lambda)} \) to have \RD (resp.\
polynomial growth). Albeit still non-trivial, this condition is much easier to
check in some concrete cases than the condition of matched pair of length
functions used in Theorem~\ref{theo:ad139c52be8424d9} and
Theorem~\ref{theo:66c73d7c75f29067}, which we shall exploit in
\S~\ref{sec:bd435170e9c7be48}. Theorem~\ref{theo:54211248e165a397} is of course
a much stronger result, provided the assumption there is satisfied, as it
completely gets rid of problem of compatibility of the relevant length
functions. This leads to the concrete examples of
\( \widehat{\Gamma \bowtie (G \rtimes \Lambda)} \) having \RD but not polynomial
growth (\S~\ref{sec:c9560fd86a3f44db}). However, it has the limitation of not
covering many other interesting examples, as will be demonstrated again in
\S~\ref{sec:bd435170e9c7be48}. It is in the opinion of the author that a more
unified approach is desired that relaxes the assumption of
Theorem~\ref{theo:54211248e165a397} while still includes the examples in
\S~\ref{sec:bd435170e9c7be48}.

This paper is organized as follows. In the long preliminary
\S~\ref{sec:46a24bbbeda06e15}, we provide all the needed background on the
representation theory that can be rather involved. Besides the purpose of making
the paper reasonably self-contained, it also serves to fix notations for later
use. In \S~\ref{sec:e9d9394edcf4dfb6}, we briefly give a conceptual and still
elementary explanation of how one can obtain non-trivial matched pair by
twisting semidirect products by a finite subgroup, followed by
\S~\ref{sec:18f8efc7904f301f} which fixes even more notations that are required
to deal with the arising subtleties. After the stage is set, we prove the main
theorems in \S~\ref{sec:934056265b2fce2f} and
\S~\ref{sec:ef1c8a06826abe0a}. Then using the tools developed so far, we
constructed the desired examples using Theorem~\ref{theo:54211248e165a397} in
\S~\ref{sec:c9560fd86a3f44db}, and another class of desired examples are
constructed in \S~\ref{sec:bd435170e9c7be48}, showing the limitation of the
procedure in \S~\ref{sec:c9560fd86a3f44db}.

We finish this introduction by making some conventions.

\subsection{Conventions and notations}
\label{sec:b02a86fd1efc0e1c}

Regarding the terminologies and notations of compact quantum groups, our choices
are mainly consistent with the works of Neshveyev \& Tuset~\cite{MR3204665}
and Woronowicz~\cite{MR1616348}.

All representations and projective representations of (quantum) groups are
assumed to be unitary and finite dimensional unless otherwise stated. We view
discrete quantum groups as the duals of compact ones.

When \( \mathbb{G} \) is a compact (quantum) group, the set of equivalence
classes of irreducible representations of \( \mathbb{G} \) is denoted by
\( \operatorname{Irr}(\mathbb{G}) \), and the trivial representation of
\( \mathbb{G} \) is denoted by \( \varepsilon_{\mathbb{G}} \), or simply
\( \varepsilon \) when \( \mathbb{G} \) is clear from context. If \( u \) is a
representation of \( \mathbb{G} \), then the equivalence class of \( u \) is
denoted by \( [u] \). The dual of \( \mathbb{G} \), as a discrete quantum group,
is denoted by \( \widehat{\mathbb{G}} \). The tensor product of two
representations \( u \) and \( v \) of \( \mathbb{G} \) is denoted by
\( u \times v \), and the tensor product of classes
\( x, y \in \operatorname{Irr}(\mathbb{G}) \) by \( x \otimes y \). The space of
operators that intertwines \( u \) and \( v \) is denoted by
\( \operatorname{Mor}_{\mathbb{G}}(u, v) \).

When \( H \), \( K \) are Hilbert spaces, the notation \( B(H, K) \) means the
Banach space of all bounded linear operators from \( H \) to \( K \). The space
\( B(H, H) \) is denoted by \( B(H) \).

For a topological group \( G \), we use \( \inn(G) \) to denote the normal
subgroup of \( \aut(G) \) consisting of inner automorphisms, where \( \aut(G) \)
is the group of topological automorphisms of \( G \). The outer automorphism
group \( \out(G) \) is of course the quotient group \( \aut(G) / \inn(G) \).

\subsection{Acknowledgement}
\label{sec:07917b7d999ef37e}

This work is supported by the ANR project ANCG (No. ANR-19-CE40-0002) in France
during its main stage of preparation, as well as the National Science Center
grant (NCN) 2020/39/I/ST1/01566 in Poland in its final stage. The author would
also like to express his gratitude to Professor Pierre Fima for several useful
discussions, and to Professor Adam Skalski for helpful suggestions.

\section{Preliminaries}
\label{sec:46a24bbbeda06e15}

As preliminaries, we briefly describe the relevant results of
(\cite{MR4345210}) and (\cite{wang2019representations}), which form
the theoretical framework of this paper. To make the presentation more
self-contained, we also briefly recall the bicrossed product construction as
presented in \cite{MR3743231}, and property \RD in the discrete quantum group
setting (\cite{MR2329000}, see also \cite{MR3448333}).

\subsection{Rapid decay and polynomial growth for discrete quantum groups of Kac
  type}
\label{sec:00df8c84fb2742bf}

As the bicrossed products of matched pair of groups considered here are always
of Kac type, we only recall the relevant formulations in this slightly
simplified case.

Let \( \mathbb{H} \) be a compact quantum group of Kac type and
\( \widehat{\mathbb{H}} \) its dual (as a discrete quantum group of Kac type),
\( \ell^{\infty}(\widehat{\mathbb{H}}) \) be the \( \ell^{\infty} \)-direct sum
\( \bigoplus_{x \in \operatorname{Irr}(\mathbb{H})}^{\ell^{\infty}}B(H_{x}) \),
and \( c_{c}(\widehat{\mathbb{H}}) \) the ideal of the algebraic direct sum,
whose \( C^{\ast} \)-completion is denoted by \( c_{0}(\widehat{\mathbb{H}})
\). A length function on \( \widehat{\mathbb{H}} \) is a map
\( l : \operatorname{Irr}(\mathbb{H}) \to \mathbb{R}_{\geq 0} \) such that
\begin{enumerate*}[label=\textup{(\roman*)}]
\item \( l([\varepsilon_{\mathbb{H}}]) = 0 \);
\item \( l(x) = l(\overline{x}) \) for all
  \( x \in \operatorname{Irr}(\mathbb{H}) \);
\item \( l(z) \leq l(x) + l(y) \) if \( z \subseteq x \otimes y \).
\end{enumerate*}
In the following, \( p_{x} \) denotes the central projection of
\( c_{0}(\widehat{\mathbb{H}}) \) that corresponds to the block \( B(H_{x}) \).
Take \( a \in c_{c}(\widehat{\mathbb{H}}) \), the \textbf{Fourier transform}
\( \mathcal{F}_{\mathbb{H}}(a) \) of \( a \) is defined by
\begin{displaymath}
  \mathcal{F}_{\mathbb{H}}(a)
  = \sum_{x \in \operatorname{Irr}(\mathbb{H})} (\dim x)
  \left[(\operatorname{Tr}_{H_{x}} \otimes \id)(u^{x}(a p_{x} \otimes 1))\right]
  \in \operatorname{Pol}(\mathbb{H})
\end{displaymath}
where \( u^{x} \) is a unitary representation on \( H_{x} \), chosen (and fixed)
in the class \( x \). The Sobolev-\( 0 \)-norm \( \norm*{a}_{\mathbb{H}, 0} \)
is determined by
\begin{displaymath}
  \norm*{a}_{\mathbb{H}, 0}^{2} = \sum_{x \in \operatorname{Irr}(\mathbb{H})} (\dim x)
  \operatorname{Tr}_{H_{x}}(a^{\ast}ap_{x}).
\end{displaymath}

Given a length function \( l \) on \( \widehat{\mathbb{H}} \), consider the
element \( L = \sum_{x \in \operatorname{Irr}(\mathbb{H})} l(x) p_{x} \) that is
affiliated (in the sense of Woronowicz \cite{MR1096123}) to
\( c_{0}(\widehat{\mathbb{H}}) \). Denote the spectral projection of \( L \)
corresponding to the interval \( \interval[open right]{n}{n+1} \) by
\( q_{n} \). The pair \( \pair*{\widehat{\mathbb{H}}}{l} \) is said to have
\begin{itemize}
\item \emph{the rapid decay property} or property \RD, if there exists a
  polynomial \( P \in \mathbb{R}[X] \), such that for all \( k \in \mathbb{N} \)
  and \( a \in q_{k} c_{c}(\widehat{\mathbb{H}}) \), we have
  \( \norm*{a}_{C(\mathbb{H})} \leq P(k) \norm*{a}_{\mathbb{H}, 0} \), where
  \( C(\mathbb{H}) \);
\item \emph{polynomial growth}, if there exists a polynomial
  \( P \in \mathbb{R}[X] \), such that for all \( k \in \mathbb{N} \),
  \begin{displaymath}
    \sum_{x \in \operatorname{Irr}(\mathbb{H}), \, k \leq l(x) < k+1} {(\dim x)}^{2} \leq P(k).
  \end{displaymath}
\end{itemize}

We say that \( \widehat{\mathbb{H}} \) has \RD (resp.\ polynomial growth), if
there exists a length function \( l \) on \( \widehat{\mathbb{H}} \) such that
the pair \( \pair*{\widehat{\mathbb{H}}}{l} \) has \RD (resp.\ polynomial
growth).

From the work \cite{MR2329000}, it is known that
\( \pair*{\widehat{\mathbb{H}}}{l} \) having polynomial growth implies it having
\RD, and the converse holds if \( \mathbb{H} \) is co-amenable; it is also known
that the duals of compact connected real Lie groups have polynomial growth.

We remark also that when \( \widehat{\mathbb{H}} \) is a classical discrete
group, the above definitions reduce to the classical case considered by
Jolissaint \cite{MR943303}.

\subsection{Bicrossed products and \RD}
\label{sec:7dac8ea5d84fa4d0}

\subsubsection{Construction of the bicrossed product}
\label{sec:11274d9921840817}

We briefly describe the bicrossed product construction associated to a matched
pair of classical groups, which is the main examples to be considered here, and
refer to \cite{MR3743231} or \cite{phdthesis}*{Chapter I} for more background and
the missing details, and we mention that this construction works in the much
more general context of matched pair of locally compact quantum groups (see
\cite{MR1970242}).

A matched pair of groups \( \pair*{\Gamma}{G} \) consists of a discrete group
\( \Gamma \), a compact (compact means quasi-compact and Hausdorff in this
paper) group \( G \), together with two actions \( \alpha \) and \( \beta \),
where \( \alpha : \Gamma \curvearrowright G \) is a right action of the group
\( \Gamma \) on the underlying set of \( G \), and
\( \beta : \Gamma \curvearrowleft G \) a right action of the group \( G \) on
the underlying set of \( \Gamma \), such that
\begin{subequations}
  \begin{equation}
    \label{eq:6b5985ad92508180}
     \gamma \in \Gamma, g, h \in G, \implies
    \alpha_{\gamma}(g h) = \alpha_{\gamma}(g) \alpha_{\beta_{g}(\gamma)}(h),
  \end{equation}
  \begin{equation}
    \label{eq:a6ca9ad979bc2175}
    r, s \in \Gamma, g \in G, \implies
    \beta_{g}(rs) = \beta_{\alpha_{s}(g)}(r) \beta_{g}(s),
  \end{equation}
  \begin{equation}
    \label{eq:f5dd567fe0dac636}
    \gamma \in \Gamma, g \in G, \implies
    \alpha_{\gamma}(e_{G}) = e_{G}, \;
    \beta_{g}(e_{\Gamma}) = e_{\Gamma},
  \end{equation}
\end{subequations}
where \( e_{\Gamma} \) (resp.\ \( e_{G} \)) denotes the identity element in the
group \( \Gamma \) (resp.\ \( G \)). It is well-known that \( (\Gamma, G) \) is
a matched pair of groups if and only if \( \Gamma \) and \( G \) can
simultaneously be identified with (topological) subgroups of a locally compact
group \( H \) in such a way that \( H = \Gamma G \),
\( \Gamma \cap G = \set*{e_{H}} \). In this identification, the underlying
actions \( \alpha \), \( \beta \) are characterized by the relations
\( \gamma g = \alpha_{\gamma}(g) \beta_{g}(\gamma) \). One also has that the
actions \( \alpha \), \( \beta \) are automatically continuous, and \( \alpha \) preserves the
Haar measure of \( G \).

We now introduce some notations. Given a matched pair of groups
\( \pair*{\Gamma}{G} \) with underlying actions \( \alpha \) and \( \beta \).
To simplify notations, we often denote the right action \( \beta \) by a dot,
thus writing \( \gamma \cdot g \) instead of \( \beta_{g}(\gamma) \). In this
spirit, the orbit of \( \gamma \in \Gamma \) under \( \beta \) is written as
\( \gamma \cdot G \). By the continuity of the action \( \beta \), we see that
the set \( G_{r,s}:= \set*{g \in G \given r \cdot g = s} \) is clopen, thus its
characteristic function \( v_{r,s} \) is an element of \( C(G) \). It is obvious
that \( G_{r,s} \ne \emptyset \) if and only if \( r,s \) belong to the same
orbit, and \( G_{\gamma}:= G_{\gamma, \gamma} \) is the isotropy subgroup of
\( G \) fixing \( \gamma \in \Gamma \). As \( G_{\gamma} \) is of finite index
in \( G \) (since \( G \) compact and \( G_{\gamma} \) open), each
\( \beta \)-orbit is finite, and we denote the set of all \( \beta \)-orbits by
\( \operatorname{Orb}_{\beta} \).

The action \( \alpha : \Gamma \curvearrowright G \) induces a
\( C^{\ast} \)-dynamical system \( (C(G), \Gamma, \widetilde{\alpha}) \), where
\( C(G) \) is the algebra of continuous functions on \( G \), and
\( \widetilde{\alpha} : \Gamma \to \operatorname{Aut}(C(G)) \) is the group
homomorphism
\( \gamma \mapsto \left\{\alpha_{\gamma^{-1}}^{\ast} : \varphi \mapsto \varphi
  \circ \alpha_{\gamma^{-1}}\right\} \). On the reduced crossed product
\( C(G) \rtimes \Gamma \), there exists a unique compact quantum group structure
whose comultiplication \( \Delta \) satisfies the following conditions:
\begin{itemize}
\item \( \Delta \) restricts to the comultiplication on \( C(G) \) that is induced by
  the multiplication on \( G \);
\item
  \( \Delta(u_{\gamma}) = \sum_{\mu \in \gamma \cdot G}u_{\gamma} v_{\gamma,
    \mu} \otimes u_{\mu} \), where we use \( u_{r} \) to denote the copy of
  \( r \in \Gamma \) in the crossed product \( C(G) \rtimes \Gamma \).
\end{itemize}

We denote the compact quantum group \( (C(G) \rtimes \Gamma, \Delta) \) by
\( \Gamma \bowtie_{\alpha, \beta} G \), or simply \( \Gamma \bowtie G \),
dropping the actions \( \alpha \) and \( \beta \) when the context makes them
clear, and call this compact quantum group the \textbf{bicrossed product} of the
matched pair of groups \( \pair*{\Gamma}{G} \). The quantum group
\( \Gamma \bowtie G \) is always of Kac type, so \S~\ref{sec:00df8c84fb2742bf}
applies to \( \widehat{\Gamma \bowtie G} \).

Throughout the rest of \S~\ref{sec:7dac8ea5d84fa4d0}, we fix a matched pair of
groups \( \pair*{\Gamma}{G} \), and the underlying left action
\( \alpha : \Gamma \curvearrowright G \), as well as the right action
\( \beta : \Gamma \curvearrowleft G \).

\subsubsection{Representation theory of the bicrossed product}
\label{sec:d0a3089327fd8687}

The representation theory of \( \Gamma \bowtie G \) that we will need includes a
classification of irreducible representations, the description of the conjugate
representation (which is the contragredient representation since
\( \Gamma \bowtie G \) is of Kac type) in terms of this classification, and most
importantly, the fusion rules. The results summarized here were first obtained in
\cite{MR4345210}, with the exception of the simplified version of the fusion
rules, Theorem~\ref{theo:c8f451b8443e316a}, which is given in
\cite{phdthesis}*{Theorem I.4.9}. We refer the reader to \cite{MR4345210}*{\S{}3}
and \cite{phdthesis}*{\S{}I.2 \& \S{}I.4} for complete proofs of the results
below.  Motivational considerations of the classification of irreducible
representations can be found in \cite{phdthesis}*{\S{}I.3}.

We will formulate our results using the notion of
\( \mathscr{O} \)-representations, which we now briefly describe.  Take an orbit
\( \mathscr{O} \in \operatorname{Orb}_{\beta} \). By an
\textbf{\( \mathscr{O} \)-representation of \( G \)}, we mean a unitary operator
\( U \) in \( B(\ell^{2}(\mathscr{O})) \otimes B(H) \otimes C(G) \), where
\( H \) is some finite dimensional Hilbert space, such that if we write
\begin{displaymath}
  U = \sum_{r,s \in \mathscr{O}} e_{r,s} \otimes u_{r,s}
\end{displaymath}
with \( {(e_{r,s})} \) the canonical matrix units of
\( B(\ell^{2}(\mathscr{O})) \) and \( u_{r,s} \in B(H) \otimes C(G) \), then
(recall that \( v_{r,s} \) is the characteristic function of \( G_{r,s} \))
\begin{displaymath}
  \forall r,s \in \mathscr{O}, \;
  u^{\ast}_{r,s}u_{r,s} = u_{r,s}u^{\ast}_{r,s} = \id_{H} \otimes v_{r,s} \in B(H) \otimes C(G).
\end{displaymath}
Viewing each \( u_{r,s} \) as a map from \( G \) into \( B(H) \), then the
restriction \( u_{r,r} \vert_{G_{r}} \) is a unitary representation of
\( G_{r} \) on \( H \). We say that \( U \) is \( \mathscr{O} \)-irreducible if
\( u_{r,r}\vert_{G_{r}} \) is irreducible for some \( r \in \mathscr{O} \),
which is equivalent to \( u_{r,r}\vert_{G_{r}} \) being irreducible for all
\( r \in \mathscr{O} \).  If
\( W = \sum_{r,s \in \mathscr{O}} e_{r,s} \otimes w_{r,s} \) is another
\( \mathscr{O} \)-representation of \( G \), say on
\( \ell^{2}(\mathscr{O}) \otimes K \), we say that \( U \) and \( W \) are
\( \mathscr{O} \)-equivalent if there exists a unitary operator
\( T \in B(H, K) \) such that \( T \) intertwines \( u_{r,r} \vert_{G_{r}} \)
and \( v_{r,r}\vert_{G_{r}} \) for some \( r \in \mathscr{O} \), which is also
equivalent to \( T \) intertwines \( u_{r,r} \vert_{G_{r}} \) and
\( v_{r,r}\vert_{G_{r}} \) for all \( r \in \mathscr{O} \).

The \( \mathscr{O} \)-representations can be used to produce representations of
the bicrossed product \( \Gamma \bowtie G \). Namely, if \( U \) is an
\( \mathscr{O} \)-representation above, then the unitary operator
\begin{equation}
  \label{eq:ddd54c506170289c}
  \mathfrak{R}_{\mathscr{O}}(U)
  := \left(\sum_{\gamma \in \mathscr{O}}e_{\gamma, \gamma} \otimes
    \id_{\mathscr{H}} \otimes
    u_{\gamma} \right) U
\end{equation}
is a unitary representation of \( \Gamma \bowtie G \) on
\( \ell^{2}(\mathscr{O}) \otimes H \). For
\( \mathscr{O} \in \operatorname{Orb}_{\beta} \), denote the set of
\( \mathscr{O} \)-equivalence classes of \( \mathscr{O} \)-irreducible
representations by \( \operatorname{Irr}_{\mathscr{O}}(G) \).
\begin{theo}[Classification of irreducible representations, \cite{phdthesis}*{p26, Theorem~I.4.9~(c)}]
  \label{theo:72e7adf90e3ce2dc}
  The map
  \begin{equation}
    \label{eq:b84df145df7f4328}
    \begin{split}
      \mathfrak{R} : \coprod_{\mathscr{O} \in \operatorname{Orb}_{\beta}}
      \operatorname{Irr}_{\mathscr{O}}(G) & \to \operatorname{Irr}(\Gamma \bowtie G) \\
      [U] \in \operatorname{Irr}_{\mathscr{O}}(G) & \mapsto
      [\mathfrak{R}_{\mathscr{O}}(U)]
    \end{split}
  \end{equation}
  is a well-defined bijection.
\end{theo}

The \( \mathscr{O} \)-irreducible representations are in fact just suitable
copies of induced representations of irreducible representations of
\( G_{\gamma} \) for any \( \gamma \in \mathscr{O} \). The precise formulation of this
correspondence is given by the following proposition.

\begin{prop}
  \label{prop:d1bbc7a1a0561f1e}
  Let \( \mathscr{O} \) be a \( \beta \)\nobreakdash-orbit,
  \( \gamma \in \mathscr{O} \), and
  \( u : G_{\gamma} \to \mathcal{B}(\mathscr{H}) \) a finite dimensional unitary
  representation of \( G_{\gamma} \). Take
  \( \sigma_{\mu} \in G_{\gamma, \mu} \) for each \( \mu \in \mathscr{O} \) with
  \( \sigma_{\gamma} = e_{G} \) and define
  \begin{displaymath}
    u_{r, s}(g) =
    \begin{cases}
      u\left(\sigma_{r} g \sigma_{s}^{-1}\right), &
      \text{if }  g \in G_{r, s} ; \\
      0, & \text{if }  g \notin G_{r, s}.
    \end{cases}
  \end{displaymath}
  Then the operator
  \( U = \sum_{r, s \in \mathscr{O}} e_{r, s} \otimes u_{r, s} \) is an
  \( \mathscr{O} \)\nobreakdash-representation on
  \( \ell^{2}(\mathscr{O}) \otimes \mathscr{H} \), and
  \( U \simeq \operatorname{Ind}_{G_{\gamma}}^{G}(u) \).

  Moreover, for any \( \gamma \in \mathscr{O} \), the map
  \begin{equation}
    \label{eq:2a0bd0bc8a9cf420}
    \begin{split}
      \varPhi_{\gamma} : \operatorname{Irr}(G_{\gamma}) & \to \operatorname{Irr}_{\mathscr{O}}(G) \\
      [u] & \mapsto \left[\sum_{r,s \in \mathscr{O}} e_{r,s} \otimes
        u_{r,s}\right],
    \end{split}
  \end{equation}
  where \( u_{r,s} \) is defined above, is a well-defined (i.e.\ independent of
  the representative \( u \) of the class \( [u] \) and independent of the
  choices of \( \sigma_{\mu} \in G_{\gamma, \mu} \)) bijection.
\end{prop}

For an \( \mathscr{O} \)-representation
\( U = \sum_{r,s \in \mathscr{O}}e_{r,s} \otimes u_{r,s} \) on
\( \ell^{2}(\mathscr{O}) \otimes H \), one can show that in fact
\( u_{r,s} \in B(H) \otimes \operatorname{Pol}(G) \). Let
\( j : B(H) \to B(\overline{H}) \) be the canonical conjugate linear
anti-isomorphism, where \( \overline{H} \) is the conjugate Hilbert space of
\( H \), and let \( S \) be the antipode of the Hopf-\( \ast \)-algebra
\( \operatorname{Pol}(G) \). The matched pair relations
\eqref{eq:6b5985ad92508180}, \eqref{eq:a6ca9ad979bc2175} and
\eqref{eq:f5dd567fe0dac636} imply that
\( \mathscr{O}^{-1} = \set*{\gamma^{-1} \given \gamma \in \mathscr{O}} \) is
also a \( \beta \)-orbit. The unitary
\begin{displaymath}
  U^{\dag}:= \sum_{r,s \in \mathscr{O}} e_{s^{-1},r^{-1}} \otimes (\id \otimes \alpha_{s^{-1}}^{\ast})
  \left[(j \otimes S)u_{r, s}\right] \in B(\ell^{2}(\mathscr{O}^{-1})) \otimes B(\overline{H}) \otimes C(G)
\end{displaymath}
is an \( \mathscr{O}^{-1} \) representation of \( G \) on
\( \ell^{2}(\mathscr{O}^{-1}) \otimes \overline{H} \), which is equivalent to
the conjugate representation of \( U \).

\begin{theo}[Conjugate of irreducible representations, \cite{phdthesis}*{pp.28,29, Definition~I.4.12 \& Theorem~I.4.13}]
  \label{theo:82a7cf677459fdd2}
  The operation \( U \mapsto U^{\dag} \) preserves \( \mathscr{O} \)-equivalence
  classes, and induces a well defined involution on
  \( \coprod_{\mathscr{O} \in \operatorname{Orb}_{\beta}}
  \operatorname{Irr}_{\mathscr{O}}(G) \), and the classification map
  \( \mathfrak{R} : \coprod_{\mathscr{O} \in
    \operatorname{Orb}_{\beta}}\operatorname{Irr}_{\mathscr{O}}(G) \to
  \operatorname{Irr}(\Gamma \bowtie G) \) preserves involution, where the
  involution on the latter set is given by conjugation. Moreover, if
  \( \gamma \in \mathscr{O} \in \operatorname{Orb}_{\beta} \), and \( U \) is an
  \( \mathscr{O} \)-representation such that
  \( U \simeq \operatorname{Ind}(u) \) for some irreducible representation
  \( u \) of \( G_{\gamma} \), then \( \alpha_{\gamma^{-1}} \) restricts to an
  isomorphism of compact groups from \( G_{\gamma^{-1}} \) onto
  \( G_{\gamma} \), and
  \( U^{\dag} \simeq \operatorname{Ind}_{G_{\gamma^{-1}}}^{G}(u \circ
  \alpha_{\gamma^{-1}}) \).
\end{theo}

To describe the fusion rules of \( \Gamma \bowtie G \), we need a certain
twisted tensor product which we now describe. For \( i = 1, 2, 3 \), let
\( \mathscr{O}_{i} \in \operatorname{Orb}_{\beta} \), suppose
\( U_{i} = \sum_{r,s \in \mathscr{O}_{i}}e_{r,s} \otimes u_{r,s}^{(i)} \) is an
\( \mathscr{O}_{i} \)-representation on
\( \ell^{2}(\mathscr{O}_{i}) \otimes H_{i} \) and
\( W_{i} = \mathfrak{R}_{\mathscr{O}_{i}}(U_{i}) \). We will describe
\( \dim \operatorname{Mor}_{\Gamma \bowtie G}(W_{3}, W_{1} \times W_{2}) \),
which is slightly more general than the fusion rules (for which we only need the
case where \( U_{i} \) is \( \mathscr{O}_{i} \)-irreducible for all \( i
\)).

The product set
\begin{displaymath}
  \mathscr{O}_{1}\mathscr{O}_{2} = \set*{\gamma_{1}\gamma_{2} \given
    (\gamma_{1}, \gamma_{2}) \in \mathscr{O}_{1} \times \mathscr{O}_{2}}
\end{displaymath}
is a disjoint union of \( \beta \)-orbits. For
\( \gamma \in \mathscr{O}_{1}\mathscr{O}_{2} \), set
\begin{displaymath}
  K^{\gamma}_{\mathscr{O}_{1},\mathscr{O}_{2}} =
  \operatorname{Vect}\set*{\delta_{\gamma_{1}} \otimes \delta_{\gamma_{2}}
    \given (\gamma_{1}, \gamma_{2}) \in \mathscr{O}_{1} \times
    \mathscr{O}_{2} \text{ with } \gamma_{1}\gamma_{2} = \gamma}
  \subseteq \ell^{2}(\mathscr{O}_{1}) \otimes \ell^{2}(\mathscr{O}_{2}),
\end{displaymath}
and define a map
\( U_{1} \times_{\gamma} U_{2} : G_{\gamma} \to B(K^{\gamma}_{\mathscr{O}_{1},
  \mathscr{O}_{2}}) \) by
\begin{displaymath}
  U_{1} \times_{\gamma} U_{2} (g) :=  \sum_{r_{1}s_{1} = r_{2}s_{2} = \gamma}
  \left(e_{r_{1}, s_{1}} \otimes e_{r_{2}, s_{2}}\right)  \otimes
  u^{(1)}_{r_{1}, s_{1}}\bigl(\alpha_{r_{2}}(g)\bigr) \otimes u^{(2)}_{r_{2}, s_{2}}(g).
\end{displaymath}
One checks that \( U_{1} \times_{\gamma} U_{2} \) is a unitary representation of
\( G_{\gamma} \), called the \emph{tensor product of \( U_{1} \) and \( U_{2} \)
  twisted by \( \gamma \)}. We then have the following theorem which includes the
fusion rules of \( \Gamma \bowtie G \) as a particular case.

\begin{theo}[Fusion rules]
  \label{theo:c8f451b8443e316a}
  Using the above notations, the following hold.
  \begin{enumerate}[label=\textup{(\roman*)}]
  \item If
    \( \mathscr{O}_{3} \cap \mathscr{O}_{1} \mathscr{O}_{2} = \emptyset \), then
    \( \dim \operatorname{Mor}(W_{3}, W_{1} \times W_{2}) = 0 \);
  \item otherwise \( \mathscr{O}_{3} \subseteq \mathscr{O}_{1} \mathscr{O}_{2} \), then
    \begin{gather*}
      \dim \operatorname{Mor}(W_{3}, W_{1} \times W_{2})
      = \dim \operatorname{Mor}_{G_{\gamma}}
      \left(u^{(3)}_{\gamma,\gamma}\vert_{G_{\gamma}}, U_{1} \times_{\gamma} U_{2}\right) \\
      = \dim \operatorname{Mor}_{G_{\mu}}\left(u^{(3)}_{\mu,\mu}\vert_{G_{\mu}}, U_{1} \times_{\mu} U_{2}\right)
    \end{gather*}
    for all \( \gamma, \mu \in \mathscr{O}_{3} \).
  \end{enumerate}
\end{theo}

\begin{rema}
  \label{rema:ad10f20261ae38e8}
  As mentioned at the beginning of \S~\ref{sec:d0a3089327fd8687}, Theorem
  \ref{theo:c8f451b8443e316a} simplifies our previous result
  \cite{MR4345210}*{Theorem 3.2}.
\end{rema}

\subsubsection{Permanence of \RD and polynomial growth}
\label{sec:7c3628ba8d3c35e5}

Intuitively speaking, the (dual of) the bicrossed product \( \Gamma \bowtie G \)
has \RD (resp.\ polynomial growth) if and only both \( \Gamma \) and
\( \widehat{G} \) have the same property in a compatible manner. Here the
precise formulation of this compatibility requires the notion of matched pair of
length functions, which arise naturally from the representation theory of
\( \Gamma \bowtie G \) as described in \S~\ref{sec:d0a3089327fd8687}. Again, the
following results were first obtained in \cite{MR4345210} (see also
\cite{phdthesis}*{Chapter I} for a more polished treatment).

\begin{defi}
  \label{defi:46e581e887d239e1}
  Let \( l_{\Gamma} \) be a length function on \( \Gamma \),
  \( l_{\widehat{G}} \) a length function on \( \widehat{G} \). The pair
  \( \pair*{l_{\Gamma}}{l_{\widehat{G}}} \) is said to be matched, if there are
  a family of maps
  \begin{displaymath}
    \set*{l_{\mathscr{O}} : \operatorname{Irr}_{\mathscr{O}}(G) \to
      \mathbb{R}_{\geq0}\given \mathscr{O} \in \operatorname{Orb}_{\beta}}
  \end{displaymath}
  indexed by \( \operatorname{Orb}_{\beta} \), such that the following
  conditions are satisfied:
  \begin{itemize}
  \item     \( l_{\set*{e_{\Gamma}}}([\varepsilon_{G}]) = 0 \);
  \item for all \( \mathscr{O} \in \operatorname{Orb}_{\beta} \) and
    \( [U] \in \operatorname{Irr}_{\mathscr{O}}(G) \), we have
    \( l_{\mathscr{O}}([U]) = l_{\mathscr{O}^{-1}}([U^{\dag}]) \);
  \item for \( i = 1, 2, 3 \), let
    \( \mathscr{O}_{i} \in \operatorname{Orb}_{\beta} \), and
    \( [U_{i}] \in \operatorname{Irr}_{\mathscr{O}_{i}}(G) \), with
    \( U_{i} = \sum_{r,s \in \mathscr{O}_{i}} e_{r,s} \otimes u^{(i)}_{r,s} \)
    being an \( \mathscr{O}_{i} \)\nobreakdash-irreducible
    \( \mathscr{O}_{i} \)\nobreakdash-representation of \( G \) on
    \( \ell^{2}(\mathscr{O}_{i}) \otimes \mathscr{H}_{i} \), if
    \begin{displaymath}
      \dim \operatorname{Mor}_{G_{\gamma}}\left(u^{(3)}_{\gamma,\gamma}
        \vert_{G_{\gamma}}, U_{1} \times_{\gamma} U_{2}\right) \neq 0
    \end{displaymath}
    for some (hence for all, by Theorem~\ref{theo:c8f451b8443e316a})
    \( \gamma \in \mathscr{O}_{3} \), then
    \begin{displaymath}
      l_{\mathscr{O}_{3}}([U_{3}]) \leq l_{\mathscr{O}_{1}}([U_{1}]) +
      l_{\mathscr{O}_{2}}([U_{2}]);
    \end{displaymath}
  \item for all
    \( [U] \in \operatorname{Irr}_{\set*{e_{\Gamma}}}(G) = \operatorname{Irr}(G)
    \), we have \( l_{\widehat{G}}([U]) = l_{\set*{e_{\Gamma}}}([U]) \) ;
  \item for all \( \mathscr{O} \in \operatorname{Orb}_{\beta} \), the image
    \( l_{\Gamma}(\mathscr{O}) \) is the singleton
    \( l_{\mathscr{O}}([\varepsilon_{\mathscr{O}}]) \), where
    \( \varepsilon_{\mathscr{O}}:= \sum_{r,s \in \mathscr{O}}e_{r,s} \otimes
    v_{r,s} \) is the trivial \( \mathscr{O} \)-representation (in particular,
    \( l_{\Gamma} \) is \( \beta \)-invariant).
  \end{itemize}

  If this is the case, we say that the family
  \( \set*{l_{\mathscr{O}} \given \mathscr{O} \in \operatorname{Orb}_{\beta}} \)
  is \textbf{affording} for the matched pair
  \( \pair*{l_{\Gamma}}{l_{\widehat{G}}} \).
\end{defi}

\begin{theo}[Permanence of polynomial growth]
  \label{theo:ad139c52be8424d9}
  The following are equivalent:
  \begin{enumerate}[label=\textup{(\roman*)}]
  \item \( \widehat{\Gamma \bowtie G} \) has polynomial growth;
  \item there exists a matched pair of length functions
    \( \pair*{l_{\widehat{G}}}{l_{\Gamma}} \), such that both
    \( \pair*{\widehat{G}}{l_{\widehat{G}}} \) and
    \( \pair*{\Gamma}{l_{\Gamma}} \) have polynomial growth.
  \end{enumerate}
\end{theo}

\begin{theo}[Permanence of \RD]
  \label{theo:66c73d7c75f29067}
  The following are equivalent:
  \begin{enumerate}[label=\textup{(\roman*)}]
  \item \( \widehat{\Gamma \bowtie G} \) has \RD;
  \item there exists a matched pair of length functions
    \( \pair*{l_{\widehat{G}}}{l_{\Gamma}} \), such that
    \( \pair*{\widehat{G}}{l_{\widehat{G}}} \) has polynomial growth and
    \( \pair*{\Gamma}{l_{\Gamma}} \) has \RD;
  \item there exists a matched pair of length functions
    \( \pair*{l_{\widehat{G}}}{l_{\Gamma}} \), such that both
    \( \pair*{\widehat{G}}{l_{\widehat{G}}} \) and
    \( \pair*{\Gamma}{l_{\Gamma}} \) have \RD.
  \end{enumerate}
\end{theo}

Note that since duals of compact groups are amenable, they have polynomial growth if and only if they have \RD (\cite{MR2329000}).

\begin{rema}
  \label{rema:c83c0db13b907b8b}
  When trying to construct concrete examples, the theoretical characterizations
  given in Theorems~\ref{theo:ad139c52be8424d9} and \ref{theo:66c73d7c75f29067}
  have a serious drawback: it is in general very hard to verify that the length
  functions on \( \Gamma \) and \( \widehat{G} \) that one wants to use are
  actually matched (to determine the existence of an affording family is quite
  tricky).
\end{rema}

\subsection{Representation theory of some semi-direct products}
\label{sec:cc23f784f4d89121}

For our purposes of constructing length functions on the dual, we need a good
understanding of representations of semi-direct products of the form
\( G \rtimes \Lambda \), where \( G \) is a compact (meaning quasi-compact and
Hausdorff) group and \( \Lambda \) is a finite group acting on \( G \) by
topological automorphisms. The more general results in
\cite{wang2019representations} applies here by taking the compact quantum group
\( \mathbb{G} \) studied there to be the classical compact group \( G \). In
this case, the irreducible representations can already be described by Mackey's
method \cite{MR0031489}, \cite{MR0098328}. But even in this special case, the
result on fusion rules that we will present seems fairly recent to the best of
our knowledge. We refer our reader to \cite{wang2019representations} for more
details and complete proofs. We will also freely use the Peter-Weyl theory of
projective representations of finite groups (see e.g.\ \cite{MR3299063}).

Throughout the rest of \S~\ref{sec:cc23f784f4d89121}, we fix a compact group
\( G \), a finite group \( \Lambda \), and a group homomorphism
\( \alpha : \Lambda \to \operatorname{Aut}(G) \), seen as \( \Lambda \) acting
on \( G \) via topological automorphisms, and \( G \rtimes \Lambda \) denotes
the semi-direct with respect to this action. As a topological space,
\( G \rtimes \Lambda \) is the Cartesian product \( G \times \Lambda \) where
\( \Lambda \) is equipped with the discrete topology. The multiplication on
\( G \rtimes \Lambda \) is given by \( (g, r)(h, s) = (g \alpha_{r}(h), r s) \),
which makes \( G \rtimes \Lambda \) a compact group.

\subsubsection{Classification of irreducible representations}
\label{sec:2a0c20223a216ec6}

The action \( \alpha : \Lambda \curvearrowright G \) by topological
automorphisms induces an action
\( \widetilde{\alpha} : \Lambda \curvearrowright \operatorname{Irr}(G) \) via
\( \lambda\cdot [u] := [u \circ \alpha_{\lambda}^{-1}] \), which comes from the
action of \( \Lambda \) on the (proper) class of all representations of \( G \)
by \( \lambda \cdot u := u \circ \alpha_{\lambda}^{-1} \).  For each
\( x \in \operatorname{Irr}(G) \), one can associate a (unitary) projective
representation \( V \) of the isotropy subgroup \( \Lambda_{x} \) fixing
\( x \), as follows. Take any representation \( u \in x \) on some Hilbert space
\( H \). By definition, \( \lambda \cdot u \simeq u \) if and only if
\( \lambda \in \Lambda_{x} \). Hence for any \( \lambda \in \Lambda_{x} \),
there exists a unitary operator, uniquely determined up to a constant in the
circle group \( \mathbb{T} = \set*{z \in \mathbb{C} \given \abs*{z} = 1} \),
such that \( V(\lambda) \in \operatorname{Mor}_{G}(\lambda \cdot u, u) \). As
\( \operatorname{Mor}_{G}(\lambda \cdot u, u) =
\operatorname{Mor}_{G}(\lambda\mu \cdot u, \mu \cdot u) \) for all
\( \lambda, \mu \in \Lambda_{x} \), one checks that
\( V(\lambda\mu){V(\mu)}^{\ast}{V(\lambda)}^{\ast} \in \operatorname{Mor}(u, u)
\), hence is a multiple of \( \id_{H} \) by a scalar in \( \mathbb{T} \). Thus
\( V : \Lambda_{x} \to B(H) \), \( \lambda \mapsto V(\lambda) \) is a projective
representation of \( \Lambda_{x} \) on \( H \). The projective representation
\( V \) is uniquely determined by the class \( x \in \operatorname{Irr}(G) \) in
the following sense: if \( V' \) is another projective representation of
\( \Lambda_{x} \) such that
\( V'(\lambda) \in \operatorname{Mor}(\lambda \cdot v, v) \) for some
\( v \in x \) and all \( \lambda \in \Lambda_{x} \), then for any unitary
equivalence \( U \in \operatorname{Mor}(u, v) \), there exists a unique map
\( \mathsf{b} : \Lambda_{x} \to \mathbb{T} \), such that for all
\( \lambda \in \Lambda_{x} \), we have
\( V(\lambda) = \mathsf{b}(\lambda) U^{\ast}V'(\lambda)U \). This means that the
cohomology class \( [\omega] \) in \( H^{2}(\Lambda_{x}, \mathbb{T}) \) (the
second group cohomology of \( \Lambda_{x} \) with coefficients in
\( \mathbb{T} \) as a trivial \( \Lambda_{x} \)-module) of the cocycle
\( \omega \) of the projective representation \( V \) depends only on
\( x \in \operatorname{Irr}(G) \).  Obviously, we can replace \( \Lambda_{x} \)
by any of its subgroup in the above discussion.

\begin{nota}
  \label{nota:e382ab22fd59faf5}
  By \( \mathcal{G}_{\mathrm{iso}}(\Lambda) \), we mean the set of finite
  intersections of the isotropy subgroups \( \Lambda_{x} \),
  \( x \in \operatorname{Irr}(G) \).
\end{nota}

\begin{defi}
  \label{defi:82f41374ee1cee2e}
  Take any \( \Lambda_{0} \in \mathcal{G}_{\mathrm{iso}} \), a
  \textbf{generalized representation parameter} (abbreviated as \textbf{GRP}
  later) associated with \( \Lambda_{0} \) (for the semidirect product
  \( G \rtimes \Lambda \)) is a triple \( (u, V, v) \), where \( u \) is a
  representation of \( G \) on some finite dimensional Hilbert space \( H \),
  \( V \) is a projective representation of \( \Lambda_{0} \) on the same
  \( H \) with \( V(\lambda) \in \operatorname{Mor}_{G}(\lambda \cdot u, u) \)
  for each \( \lambda \in \Lambda_{0} \), and \( v \) is a projective
  representation of \( \Lambda_{0} \) on a possibly different finite dimensional
  Hilbert space \( K \), such that the cocycle of \( v \) is the opposite of
  that of \( V \). The GRP \( (u, V, v) \) is called a \textbf{representation
    parameter} (abbreviated as \textbf{RP} later) if \( u \) is irreducible, and
  if in addition, \( \Lambda_{0} = \Lambda_{[u]} \), we say that \( (u, V, v) \)
  is a \textbf{distinguished representation parameter} (abbreviated as
  \textbf{DRP} in the following).
\end{defi}

Given a RP \( (u, V, v) \) associated with \( \Lambda_{0} \) as in the
definition above, one can associate an irreducible representation
\( \mathscr{R}_{\Lambda_{0}}(u, V, v) \) of \( G \rtimes \Lambda_{0} \) (which
is a subgroup of \( G \rtimes \Lambda \) in which we are interested), called
\textbf{the representation of} \( G \rtimes \Lambda_{0} \) \textbf{parameterized
  by} \( (u, V, v) \), as follows: the carrier space of this representation is
\( H \otimes K \), and as a map from \( G \rtimes \Lambda_{0} \) to
\( \mathcal{U}(H \otimes K) \), the representation
\( \mathscr{R}_{\Lambda_{0}}(u, V, v) \) sends
\( (g, \lambda) \in G \rtimes \Lambda_{0} \) to
\( u(g)V(\lambda) \otimes v(\lambda) \). Now the induced representation
\( \mathscr{R}(u, V, v):= \operatorname{Ind}_{G \rtimes \Lambda_{0}}^{G \rtimes
  \Lambda}\bigl(\mathscr{R}_{\Lambda_{0}}(u, V, v)\bigr) \) is called
\textbf{the representation of} \( G \rtimes \Lambda \) \textbf{parameterized by}
\( (u, V, v) \). If \( (u, V, v) \) is distinguished, i.e.\ if \( (u, V, v) \)
is a DRP, then \( \mathscr{R}(u, V, v) \) is irreducible.

Fix a \( \Lambda_{0} \in \mathcal{G}_{\mathrm{iso}} \), we say two DRP
associated with \( \Lambda_{0} \), namely \( (u_{1},V_{1},v_{1}) \) and
\( (u_{2}, V_{2}, v_{2}) \), are \textbf{equivalent}, if
\( \mathscr{R}_{\Lambda_{0}}(u_{1}, V_{1}, v_{1}) \) and
\( \mathscr{R}_{\Lambda_{0}}(u_{2}, V_{2}, v_{2}) \) are equivalent. This
equivalence relation can be characterized more concretely as the equivalence of
the following conditions:
\begin{enumerate}[label=\textup{(\alph*)}]
\item \( (u_{1},V_{1},v_{1}) \) and \( (u_{2},V_{2},v_{2}) \) are equivalent
  DRPs;
\item there exists a mapping \( \mathsf{b}: \Lambda_{0} \to \mathbb{T} \) such
  that the projective representations \( \mathsf{b}V_{1} \) and \( V_{2} \)
  share the same cocycle, and both
  \( \operatorname{Mor}_{G}(u_{1},u_{2}) \cap
  \operatorname{Mor}_{\Lambda_{0}}(\mathsf{b}V_{1}, V_{2}) \) and
  \( \operatorname{Mor}_{\Lambda_{0}}(v_{1}, v_{2}) \) are non-zero;
\item there exists a mapping \( \mathsf{b}: \Lambda_{0} \to \mathbb{T} \) such
  that the projective representations \( \mathsf{b}V_{1} \) and \( V_{2} \)
  share the same cocycle, and both
  \( \operatorname{Mor}_{G}(u_{1},u_{2}) \cap
  \operatorname{Mor}_{\Lambda_{0}}(\mathsf{b}V_{1}, V_{2}) \) and
  \( \operatorname{Mor}_{\Lambda_{0}}(v_{1}, v_{2}) \) contain unitary
  operators.
\end{enumerate}

\begin{nota}
  \label{nota:d4dd198b99ae5dc6}
  For each \( \Lambda_{0} \in \mathcal{G}_{\mathrm{iso}} \), we denote the set
  of equivalence classes of DRPs associated with \( \Lambda_{0} \) by
  \( \mathfrak{D}_{\Lambda_{0}} \), and we denote by \( \mathfrak{D} \) the
  union of \( \mathfrak{D}_{\Lambda_{x}} \), \( x \in \operatorname{Irr}(G) \).
\end{nota}

Given a GRP (resp.\ RP, resp.\ DRP) \( \mathscr{D}:= (u, V, v) \), the
componentwise contragredient \( (u^{c}, V^{c}, v^{c}) \) is still a GRP (resp.\
RP, resp.\ DRP), called the \textbf{contragredient} of \( \mathscr{D} \), and is
denoted by \( \mathscr{D}^{c} \). Moreover, for each \( r \in \Lambda \), denote
by \( \operatorname{Ad}_{r} \) the inner automorphisms \( x \mapsto rxr^{-1} \)
of \( \Lambda \), then the triple
\begin{displaymath}
  r \cdot \mathscr{D} := (r \cdot u, r \cdot V, r \cdot v)
  = (u \circ \alpha_{r}^{-1}, V \circ \operatorname{Ad}_{r}^{-1}\vert_{r\Lambda_{0}r^{-1}}, v \circ \operatorname{Ad}_{r}^{-1}\vert_{r\Lambda_{0}r^{-1}})
\end{displaymath}
is a GRP (resp.\ RP, resp.\ DRP) associated with the subgroup
\( r\Lambda_{0}r^{-1} \in \mathcal{G}_{\mathrm{iso}} \). One checks easily that
\( r \cdot [\mathscr{D}]:= [r \cdot \mathscr{D}] \) gives an action of
\( \Lambda \) on \( \mathscr{D} \), and this action preserves contragredients.

\begin{theo}[Classification of irreducible representation of
  \( G \rtimes \Lambda \)]
  \label{theo:f1781e3972950519}
  The mapping
  \begin{align*}
    \Psi: \mathfrak{D} &\to \operatorname{Irr}(G \rtimes \Lambda) \\
    \mathscr{D}:= [(u, V, v)] \in \mathfrak{D}_{\Lambda_{0}}%
                       & \mapsto \Psi_{\Lambda_{0}}(\mathscr{D}):= [\mathscr{R}(u, V, v)]
  \end{align*}
  is a well-defined surjection whose fibers are exactly the orbits of the action
  \( \Lambda \curvearrowright \mathfrak{D} \). Moreover,
  \( \Psi_{\Lambda_{0}} \), hence in particular \( \Psi \), preserves
  contragredients.
\end{theo}

\subsubsection{The fusion rules}
\label{sec:ab718acd8968aed8}

We give a description of the fusion rules for \( G \rtimes \Lambda \) with the
help of a seemingly nontrivial reduction process.

For \( i = 1, 2, 3 \), consider an irreducible representation \( W_{i} \) of
\( G \rtimes \Lambda \) parameterized by some DRP
\( \mathscr{D}_{i}:= (u_{i}, V_{i}, v_{i}) \) associated with the isotropy
subgroup \( \Lambda_{i}:= \Lambda_{[u_{i}]} \). By
Theorem~\ref{theo:f1781e3972950519}, one only needs to find a more or less
explicit formula for
\begin{displaymath}
  \dim \operatorname{Mor}_{G \rtimes \Lambda}(W_{1}, W_{2} \times W_{3})
\end{displaymath}
in order to obtain a description of the desired fusion rules.

We now describe the reduction procedure mentioned above. One checks easily that
\( (u_{2} \times u_{3}, V_{2} \times V_{3}, v_{2} \times v_{3}) \) is a GRP
associated with the subgroup
\( \Lambda_{0}:= \cap_{i=1}^{3}\Lambda_{i} \in \mathcal{G}_{\mathrm{iso}} \). To
fix the notations, let \( H_{i} \) (resp.\ \( K_{i} \)) be the carrier space of
\( u_{i} \) (resp.\ \( v_{i} \)). Then there exists a unique subrepresentation
\( u_{p} \) of \( u_{2} \times u_{3} \) determined by a unique subspace
\( H_{p} \) of \( H_{1} \otimes H_{2} \), such that \( u_{p} \) is maximal among
the subrepresentations of \( u_{2} \times u_{3} \) that are equivalent to
multiples of \( u_{1} \), and we denote by \( n \) the multiplicity
\( \dim \operatorname{Mor}_{G}(u_{1}, u_{p}) \) of \( u_{1} \) in \( u_{p} \)
(which of course could be \( 0 \)).

\begin{lemm}
  \label{lemm:ac666372e8d2b33a}
  The following hold.
  \begin{enumerate}[label=\textup{(\arabic*)}]
  \item The subspace \( H_{p} \) is invariant under \( V \), and determines a
    projective subrepresentation \( V_{p} \) of \( V \), such that
    \( V_{p}(\lambda) \in \operatorname{Mor}(\lambda \cdot u_{p}, u_{p}) \) for
    all \( \lambda \in \Lambda_{0} \).
  \item There exists a unique projective representation \( V'_{p} \) of
    \( \mathbb{C}^{n} \), such that \( V_{p} \) is equivalent to
    \( V_{1} \times V'_{p} \) and the projective representations
    \( V'_{p} \times v_{2} \times v_{3} \) and \( v_{1} \) of \( \Lambda_{0} \)
    share the same cocycle.
  \end{enumerate}
\end{lemm}

\begin{defi}
  \label{defi:0cda6750afc0777a}
  We call
  \( \dim \operatorname{Mor}_{\Lambda_{0}}(v_{1}, V'_{p} \times v_{2} \times
  v_{3}) \) the \textbf{incidence number} of the triple
  \( (\mathscr{D}_{1}, \mathscr{D}_{2}, \mathscr{D}_{3}) \) of DRPs.
\end{defi}

\begin{nota}
  \label{nota:5cb67e3471d223c1}
  For all \( r_{i} \in \Lambda \), \( i = 1,2,3 \), we denote by
  \( m(r_{1}, r_{2}, r_{3}) \) the incidence number of
  \( (\mathscr{D}_{1}, \mathscr{D}_{2}, \mathscr{D}_{3}) \). It is easy to check
  that \( m(r_{1},r_{2},r_{3}) = m(s_{1},s_{2},s_{3}) \) if
  \( r_{i}\Lambda_{i} = s_{i}\Lambda_{i} \), i.e.\ the incidence numbers
  \( m(z_{1}, z_{2}, z_{3}):= m(r_{1}, r_{2}, r_{3}) \) with \( r_{i} \) in the
  left coset \( z_{i} \in \Lambda / \Lambda_{i} \) is well-defined, i.e.\
  independent of the choice of the representative \( r_{i} \in z_{i} \),
  \( i = 1,2,3 \). Moreover, denote by \( \Lambda(z_{1},z_{2},z_{3}) \) the
  intersection \( \cap_{i=1}^{3}r_{i}\Lambda_{i}r_{i}^{-1} \), which is again
  independent of the choices \( r_{i} \in z_{i} \), \( i = 1,2,3 \).
\end{nota}

\begin{theo}
  \label{theo:5cf848969c56b608}
  The fusion rules for \( G \rtimes \Lambda \) is given by the following formula
  \begin{displaymath}
    \dim \operatorname{Mor}(W_{1}, W_{2} \times W_{3})
    = \sum_{z_{1} \in \Lambda / \Lambda_{1}}
    \sum_{z_{2} \in \Lambda / \Lambda_{2}}
    \sum_{z_{3} \in \Lambda / \Lambda_{3}}
    \frac{m(z_{1}, z_{2}, z_{3})}{[\Lambda: \Lambda(z_{1}, z_{2}, z_{3})]},
  \end{displaymath}
  where \( [\Lambda: \Lambda_{0}] \) denotes the index of the subgroup
  \( \Lambda_{0} \) of \( \Lambda \).
\end{theo}

\section{Nontrivial bicrossed product from semidirect product}
\label{sec:e9d9394edcf4dfb6}

Let \( G \) be a compact group, \( \Gamma \) a discrete group acting on \( G \)
via topological automorphisms given by a group morphism
\( \tau : \Gamma \to \aut(G) \). Using these data, one can form the semidirect
product \( G \rtimes_{\tau} \Gamma \), which is a locally compact group whose
underlying topological space is the topological product \( G \times \Gamma \),
and whose group law is given by
\begin{equation}
  \label{eq:fd5ca7ce6e2dc98c}
  ( g_{1}, \gamma_{1} ) ( g_{2}, \gamma_{2} ) %
  = \left( g_{1} \tau_{\gamma_{1}}(g_{2}), \gamma_{1} \gamma_{2} \right).
\end{equation}

Note that the insertion
\( \iota_{\Gamma} : \Gamma \to G \rtimes_{\tau} \Gamma \),
\( \gamma \mapsto ( e_{G}, \gamma ) \) (\( e_{G} \) denotes the identity in
\( G \)) is a group morphism. In particular, the mapping
\begin{equation}
  \label{eq:2ae50e2abed7ccb4}
  \begin{split}
    \theta : \Gamma & \to \aut\left(G \rtimes_{\tau} \Gamma\right) \\
    \gamma & \mapsto ( \adj \circ \iota_{\Gamma} )( \gamma ) = \adj_{( e_{G},
      \gamma )}
  \end{split}
\end{equation}
is a group morphism. For all \( (g, r) \in G \rtimes_{\tau} \Gamma \) and
\( \gamma \in \Gamma \), we have
\begin{equation}
  \label{eq:95d755a8d05ffa2b}
  ( e_{G}, \gamma ) (g, r) {( e_{G}, \gamma  )}^{-1} %
  = \left( \tau_{\gamma}(g), \gamma r \right)
  \left( e_{G}, \gamma^{-1} \right) %
  = \left( \tau_{\gamma}(g), \gamma r \gamma^{-1} \right).
\end{equation}
Thus as a map from the set \( G \times \Gamma \) to itself, we have
\begin{equation}
  \label{eq:a20fdc545f3d2677}
  \theta_{\gamma} := \theta( \gamma ) = \tau_{\gamma} \times \adj_{\gamma} :
  G \times \Gamma \to G \times \Gamma.
\end{equation}

Now consider any \emph{finite} subgroup \( \Lambda \) of \( \Gamma \). The group
morphism \( \theta \) defined in \eqref{eq:2ae50e2abed7ccb4} restricts to the
subgroup \( \Lambda \) to give an action
\( \Lambda \curvearrowright G \rtimes_{\tau} \Gamma \) by topological
automorphisms. This allows us to form yet another semidirect product
\( (G \rtimes_{a} \Gamma) \rtimes_{\theta} \Lambda \), whose underlying
topological space is \( G \times \Gamma \times \Lambda \). It is clear that the
group law on \( (G \rtimes_{\tau} \Gamma) \rtimes_{\theta} \Lambda \) is given
by
\begin{equation}
  \label{eq:e392544b23694d1e}
  \begin{split}
    & \leadmathskip (g_{1}, \gamma_{1}, r_{1})(g_{2}, \gamma_{2}, r_{2}) %
    = \bigl((g_{1}, \gamma_{1}) \theta_{r_{1}}(g_{2}, \gamma_{2}),
    r_{1}r_{2}\bigr) \\
    &= \Bigl((g_{1}, \gamma_{1})
    \bigl( \tau_{r_{1}}(g_{2}), r_{1}\gamma_{2}r_{1}^{-1} \bigr),
    r_{1}r_{2} \Bigr) \\
    &= \bigl(g_{1} \tau_{\gamma_{1}r_{1}}(g_{2}),
    \gamma_{1}r_{1}\gamma_{2}r_{1}^{-1}, r_{1}r_{2}\bigr).
  \end{split}
\end{equation}
By \eqref{eq:e392544b23694d1e}, both the mapping
\begin{equation}
  \label{eq:816d9f65669dad36}
  \begin{split}
    \iota_{1,3} : G \rtimes_{\tau} \Lambda  %
    &\to (G \rtimes_{\tau} \Gamma) \rtimes_{\theta} \Lambda \\
    (g, r) & \mapsto (g, e, r),
  \end{split}
\end{equation}
and
\begin{equation}
  \label{eq:932dd123b2df2201}
  \begin{split}
    \iota_{2} : \Gamma %
    & \to (G \rtimes_{\tau} \Gamma) \rtimes_{\theta} \Lambda \\
    \gamma &\mapsto (e_{G}, \gamma, e)
  \end{split}
\end{equation}
are injective group morphisms, such that for all \( \gamma \in \Gamma \),
\( (g,r) \in G \rtimes_{\tau} \Lambda \), we have
\begin{equation}
  \label{eq:010a36418b081716}
  \forall g \in G, \, \gamma, r \in \Gamma, \quad
  \iota_{1,3}(g, r) \iota_{2}(r^{-1}\gamma r) = (g, e, r)
  (e_{G}, r^{-1}\gamma r, e) %
  = (g, \gamma, r),
\end{equation}
which implies that
\begin{equation}
  \label{eq:c1b83492743362ba}
  \iota_{1,3}(G \rtimes_{\tau} \Lambda) \iota_{2}(\Gamma)
  = (G \rtimes_{\tau} \Gamma) \rtimes_{\theta} \Lambda.
\end{equation}
It is clear that
\begin{equation}
  \label{eq:4b0a583b32a913fa}
  \iota_{1,3}(G \rtimes_{\tau} \Lambda) \cap \iota_{2}(\Gamma)
  = \set*{(e_{G}, e, e)}.
\end{equation}
Moreover,
\begin{equation}
  \label{eq:0b444f0c31403c29}
  \begin{split}
    \forall g \in G,\, \gamma, r \in \Gamma, %
    &\leadmathskip \iota_{2}(\gamma) \iota_{1,3}(g,r) %
    = (e_{G}, \gamma, e) (g, e, r) \\
    &= \bigl( \tau_{\gamma}(g), \gamma, r \bigr) %
    = \bigl( \tau_{\gamma}(g), e, r \bigr) (e_{G}, r^{-1}\gamma r, e) \\
    &= \iota_{1,3}(g, r) \iota_{2}(r^{-1}\gamma r).
  \end{split}
\end{equation}

\begin{prop}
  \label{prop:71acce0cf99cd889}
  Let \( \Gamma \) be a discrete group, \( G \) a compact group, and
  \( \tau : \Gamma \to \aut(G) \) a left action of \( \Gamma \) on \( G \) by
  topological automorphisms. If \( \Lambda \) is a \emph{finite} subgroup of
  \( \Gamma \), then \( \pair*{\Gamma}{G \rtimes_{\tau} \Lambda} \) is a matched
  pair of groups with left action
  \begin{equation}
    \label{eq:4dd73304f9bbd58d}
    \begin{split}
      \alpha^{\Lambda} : \Gamma \times (G \rtimes_{\tau} \Lambda) %
      & \to G \rtimes_{\tau} \Lambda \\
      \bigl(\gamma, (g, r)\bigr) %
      & \mapsto \bigl( \tau_{\gamma}(g), r \bigr),
    \end{split}
  \end{equation}
  and right action
  \begin{equation}
    \label{eq:f5beba9cbb7ec1dd}
    \begin{split}
      \beta^{\Lambda} : \Gamma \times (G \rtimes_{\tau} \Lambda) %
      & \to \Gamma \\
      \bigl(\gamma, (g, r)\bigr) %
      & \mapsto r^{-1}\gamma r.
    \end{split}
  \end{equation}
  Moreover, the following hold.
  \begin{enumerate}
  \item \label{item:798330b3b2b36c65} The action \( \alpha^{\Lambda} \) is
    trivial if and only if \( \tau \) is trivial;
  \item \label{item:7c33a76924999328} The action \( \beta^{\Lambda} \) is
    trivial if and only if \( \Lambda \subseteq Z(\Gamma) \), where
    \( Z(\Gamma) \) is the centre of \( \Gamma \).
  \end{enumerate}
\end{prop}
\begin{proof}
  That \( \pair*{\Gamma}{G \rtimes_{\tau} \Lambda} \) is a matched pair with the
  actions \( \alpha^{\Lambda} \) and \( \beta^{\Lambda} \) follows from
  \eqref{eq:c1b83492743362ba}, \eqref{eq:4b0a583b32a913fa} and
  \eqref{eq:0b444f0c31403c29}. \ref{item:798330b3b2b36c65} and
  \ref{item:7c33a76924999328} are direct consequences of the definition of
  \( \alpha^{\Lambda} \) and \( \beta^{\Lambda} \).
\end{proof}

\section{More notations}
\label{sec:18f8efc7904f301f}

For the convenience of our discussion, we now introduce and fix some notations
related to the bicrossed product of the matched pair
\( \pair*{\Gamma}{G \rtimes_{\tau} \Lambda} \) with the actions
\( \alpha^{\Lambda} \) and \( \beta^{\Lambda} \), as described in
\S~\ref{sec:e9d9394edcf4dfb6}.

The bicrossed product of the matched pair
\( \pair*{\Gamma}{G \rtimes_{\tau} \Lambda} \) is denoted by
\( \Gamma {}_{\alpha^{\Lambda}} \bowtie_{\beta^{\Lambda}} (G \rtimes_{\tau}
\Lambda) \). When there is no risk of confusion, we often omit the actions and
simply write \( G \rtimes \Lambda \) and
\( \Gamma \bowtie (G \rtimes \Lambda) \).  Moreover, \( \aut(G) \) denotes the
group of topological automorphism of \( G \).

The isotropy subgroup of \( G \rtimes \Lambda \) fixing \( \gamma \in \Gamma \)
with respect to the action \( \beta^{\Lambda} \) is easily seen to be
\( G \rtimes \Lambda_{\gamma} \), where
\begin{equation}
  \label{eq:756fc81603ecd400}
  \Lambda_{\gamma} := \set*{r \in \Lambda \given \gamma r = r \gamma}.
\end{equation}

For \( x \in \irr(G) \), \( \gamma \in \Gamma \), we denote the isotropy
subgroup of \( \Lambda_{\gamma} \) fixing \( x \) with respect to the action
\( \Lambda_{\gamma} \curvearrowright \irr(G) \),
\( (r, [u]) \mapsto [u \circ \tau_{r}] \) by \( \Lambda_{\gamma, x} \), i.e.\
\begin{equation}
  \label{eq:d76081fad86adf15}
  \Lambda_{\gamma, x} := \set*{r \in \Lambda_{\gamma} \given r \cdot x = x}.
\end{equation}

We also need to fix some notations from the representation theory of
\( \Gamma \bowtie (G \rtimes \Lambda) \), this is where the preliminaries
described in \S~\ref{sec:46a24bbbeda06e15} come into play.

Let \( \gamma \in \Gamma \). Suppose \( \Lambda_{0} \) is an isotropy subgroup
of \( \Lambda_{\gamma} \) with respect to the action
\( \Lambda_{\gamma} \curvearrowright \irr(G) \). Let
\( \mathfrak{D}_{\gamma, \Lambda_{0}} \) denotes the set of equivalent
distinguished representation parameters (see \S~\ref{sec:2a0c20223a216ec6})
associated with \( \Lambda_{0} \), and
\begin{equation}
  \label{eq:8952fdcb78f2c81f}
  \Psi_{\gamma, \Lambda_{0}} : \mathfrak{D}_{\gamma, \Lambda_{0}}
  \to \irr(G \rtimes \Lambda_{\gamma})
\end{equation}
be the injection used to classify irreducible unitary representations of
\( G \rtimes \Lambda_{\gamma} \) as in Theorem~\ref{theo:f1781e3972950519}. Let
\( \mathfrak{D}_{\gamma} \) be the set of equivalence classes of all
distinguished representation parameters for the semidirect product
\( G \rtimes \Lambda_{\gamma} \), we then have an action of
\( \Lambda_{\gamma} \) on the class of all distinguished representation
parameters for \( G \rtimes \Lambda_{\gamma} \), which passes to the quotient
and yields an action
\( \Lambda_{\gamma} \curvearrowright \mathfrak{D}_{\gamma} \) as described in
\S~\ref{sec:2a0c20223a216ec6}. We thus have the classification \emph{surjection}
\begin{equation}
  \label{eq:22ae8c0248b07582}
  \begin{split}
    \Psi_{\gamma} : \mathfrak{D}_{\gamma} %
    & \to \irr(G \rtimes \Lambda_{\gamma}) \\
    [(u, V, v)] \in \mathfrak{D}_{\gamma, \Lambda_{0}} %
    & \mapsto \Psi_{\gamma, \Lambda_{0}}\bigl([(u, V, v)]\bigr),
  \end{split}
\end{equation}
whose fibers are exactly the orbits for the action
\( \Lambda_{\gamma} \curvearrowright \mathfrak{D}_{\gamma} \).

When \( \gamma = e_{\Gamma} \), we then have \( \Lambda_{\gamma} = \Lambda \)
and we write \( \Psi_{e_{\Gamma}} \) simply as \( \Psi \).

We use \( \orbit_{\beta^{\Lambda}} \) to denote the set of
\( \beta^{\Lambda} \)\nobreakdash-orbits. For each
\( \mathscr{O} \in \orbit_{\beta^{\Lambda}} \), let
\( \mathfrak{R}_{\mathscr{O}} \) be the mapping from the class of
\( \mathscr{O} \)\nobreakdash-representations of \( G \rtimes \Lambda \) to the
class of finite dimensional unitary representations of the bicrossed product
\( \Gamma \bowtie (G \rtimes \Lambda) \) as in \S~\ref{sec:d0a3089327fd8687},
and let \( \irr_{\mathscr{O}}(G \rtimes \Lambda) \) denote the set of
equivalence classes of \( \mathscr{O} \)\nobreakdash-irreducible
\( \mathscr{O} \)\nobreakdash-representations. We thus have the classification
bijection
\begin{equation}
  \label{eq:d1a6514ac3409730}
  \begin{split}
    \mathfrak{R} : \coprod_{\mathscr{O} \in \orbit_{\beta^{\Lambda}}} %
    \irr_{\mathscr{O}}(G \rtimes \Lambda) %
    & \to \irr\bigl( \Gamma \bowtie (G \rtimes \Lambda) \bigr) \\
    [U] \in \irr_{\mathscr{O}}(G \rtimes \Lambda) %
    & \mapsto [\mathfrak{R}_{\mathscr{O}}(U)].
  \end{split}
\end{equation}

\section{Proof of Theorem~\ref{theo:69214233f77b0115}}
\label{sec:934056265b2fce2f}

The purpose of this section is to establish Theorem~\ref{theo:69214233f77b0115}
which gives a sufficient condition for the (dual of) bicrossed product
\( \widehat{\Gamma \bowtie (G \rtimes \Lambda)} \) to have polynomial growth
(resp.\ property \RD).

We first establish the following technical result.

\begin{lemm}
  \label{lemm:57fb1888aba4d546}
  Suppose \( l_{\widehat{G}} : \irr(G) \to \mathbb{R}_{\geq0} \) is a
  \( \Gamma \)\nobreakdash-invariant length function on \( \widehat{G} \), i.e.\
  \( l_{\widehat{G}}([u^{x} \circ \tau_{\gamma}]) = l_{\widehat{G}}(x) \)
  whenever \( \gamma \in \Gamma \), \( u^{x} \in x \in \irr(G) \), and
  \( l_{\Gamma} \) is a \( \beta^{\Lambda} \)-invariant length function on
  \( \Gamma \). Then
  \begin{align*}
    l_{\widehat{G \rtimes \Lambda}} : \irr(G \rtimes \Lambda) %
    & \to \mathbb{R}_{\geq 0} \\
    \Psi([(u, V, v)]) & \mapsto l_{\widehat{G}}([u])
  \end{align*}
  is a well-defined length function on \( \widehat{G \rtimes \Lambda} \) such
  that the pair \( \pair*{l_{\Gamma}}{l_{\widehat{G \rtimes \Lambda}}} \) is
  matched.
\end{lemm}
\begin{proof}
  The fact that \( l_{\widehat{G \rtimes \Lambda}} \) is well-defined (i.e.\
  does not depend on the choice of the distinguished representation parameter
  \( (u, V, v) \)) follows from Theorem~\ref{theo:f1781e3972950519} and the
  \( \Lambda \)-invariance of \( l_{\widehat{G}} \). We now show that the pair
  \( \pair*{l_{\Gamma}}{l_{\widehat{G \rtimes \Lambda}}} \) of length
  functions is matched.

  For all \( \mathscr{O} \in \orbit_{\beta^{\Lambda}} \), define
  \( l_{\mathscr{O}} : \irr_{\mathscr{O}}(G \rtimes \Lambda) \to
  \mathbb{R}_{\geq0} \) via the following procedure. Take any
  \( \gamma \in \mathscr{O} \), and let
  \( \varPhi_{\gamma} : \irr(G \rtimes \Lambda_{\gamma}) \to
  \irr_{\mathscr{O}}(G \rtimes \Lambda) \) be the canonical bijection as in
  Proposition~\ref{prop:d1bbc7a1a0561f1e}. To avoid over-complication of our
  notations, we often implicitly identify \( \irr(G \rtimes \Lambda_{\gamma}) \)
  with \( \irr_{\mathscr{O}}(G \rtimes \Lambda) \) via the bijection
  \( \varPhi_{\gamma} \), when doing so won't cause a risk of confusion. For all
  distinguished representation parameter \( (u, V, v) \) of
  \( G \rtimes \Lambda_{\gamma} \), let
  \begin{equation}
    \label{eq:4ec696cc821b4621}
    l_{\mathscr{O}}\left(\Psi_{\gamma}\bigl([(u, V, v)]\bigr)\right)
    := l_{\widehat{G}}([u]) + l_{\Gamma}(\gamma).
  \end{equation}
  By Theorem~\ref{theo:f1781e3972950519} again, we see that
  \eqref{eq:4ec696cc821b4621} yields a well-defined mapping
  \( l_{\mathscr{O}} : \irr_{\mathscr{O}}(G \rtimes \Lambda) \to
  \mathbb{R}_{\geq 0} \). It is clear that
  \( l_{\set*{e_{\Gamma}}} = l_{\widehat{G}} \) via the identification of
  \( \irr_{\set*{e_{\Gamma}}}(G \rtimes \Lambda) \) with
  \( \irr(G \rtimes \Lambda) \) by \( \varPhi_{e_{\Gamma}} \). Moreover, we have
  \( [\varepsilon_{\mathscr{O}}] = \Psi_{\gamma}([\varepsilon_{G},
  \varepsilon_{\Lambda_{\gamma}}, \varepsilon_{\Lambda_{\gamma}}]) \), so that
  \begin{displaymath}
    l_{\Gamma}(\gamma) = l_{\widehat{G}}([\varepsilon_{G}]) + l_{\Gamma}(\gamma)
    = l_{\mathscr{O}}([\varepsilon_{\mathscr{O}}]).
  \end{displaymath}
  Therefore, to finish the proof, it remains to show that
  \( {(l_{\mathscr{O}})}_{\mathscr{O} \in \orbit_{\beta^{\Lambda}}} \) is an
  affording family in the sense of Definition~\ref{defi:46e581e887d239e1}.  By
  definition, it is clear that
  \begin{displaymath}
    l_{\set*{e_{\Gamma}}}([\varepsilon_{G \rtimes \Lambda}])
    = l_{\widehat{G}}([\varepsilon_{G}]) + l_{\Gamma}(e_{\Gamma}) = 0.
  \end{displaymath}

  The condition \( l_{\mathscr{O}}([U]) = l_{\mathscr{O}^{-1}}([U^{\dag}]) \)
  can also be easily checked. Indeed, if \( [U] \) is given by
  \( \Psi_{\gamma}\bigl([(u, V, v)]\bigr) \), then \( [U^{\dag}] \) is given by
  \( \Psi_{\gamma^{-1}}\bigl([(\overline{u}, V^{c}, v^{c})]\bigr) \) (see
  Theorem~\ref{theo:82a7cf677459fdd2}). By \eqref{eq:4ec696cc821b4621}, we now
  have
  \begin{equation}
    \label{eq:27c9a2e0216276cd}
    l_{\mathscr{O}}([U])
    = l_{\widehat{G}}([u]) + l_{\Gamma}(\gamma)
    = l_{\widehat{G}}([\overline{u}]) + l_{\Gamma}(\gamma^{-1})
    = l_{\mathscr{O}^{-1}}([U^{\dag}]).
  \end{equation}

  By Definition~\ref{defi:46e581e887d239e1}, it remains only to establish the
  following claim.

  \noindent\underline{\textbf{Claim.}} For \( i = 1, 2, 3 \), let
  \( \mathscr{O}_{i} \in \orbit_{\beta^{\Lambda}} \), and
  \( [U_{i}] \in \irr_{\mathscr{O}_{i}}(G \rtimes \Lambda) \), with
  \begin{displaymath}
    U_{i} = \sum_{r,s \in \mathscr{O}_{i}} e_{r,s} \otimes u^{(i)}_{r,s}
  \end{displaymath}
  being an \( \mathscr{O}_{i} \)\nobreakdash-irreducible
  \( \mathscr{O}_{i} \)\nobreakdash-representation of \( G \rtimes \Lambda \) on
  \( \ell^{2}(\mathscr{O}_{i}) \otimes \mathscr{H}_{i} \). If
  \begin{equation}
    \label{eq:1303e2739229a227}
    \dim \morph_{G \rtimes \Lambda_{\gamma}}\left(u^{(3)}_{\gamma,\gamma}
      \vert_{G \rtimes \Lambda_{\gamma}},
      U_{1} \times_{\gamma} U_{2}\right) \neq
    0
  \end{equation}
  for some (hence for all) \( \gamma \in \mathscr{O}_{3} \), then
  \begin{equation}
    \label{eq:85951ab07a54d731}
    l_{\mathscr{O}_{3}}([U_{3}]) \leq l_{\mathscr{O}_{1}}([U_{1}]) +
    l_{\mathscr{O}_{2}}([U_{2}]).
  \end{equation}

  Before proving the claim, we remark that until now, only the
  \( \Lambda \)\nobreakdash-invariance of \( l_{\widehat{G}} \) is needed. The
  hypothesis that \( l_{\widehat{G}} \) is \( \Gamma \)\nobreakdash-invariant
  will play an important role in the proof of the claim as we will presently
  see.

  We now prove the claim. Suppose \( [U_{i}] \) is given by some
  \( \Psi_{\gamma_{i}}\bigl([(u_{i}, V_{i}, v_{i})]\bigr) \) and let
  \( \Lambda_{i} := \Lambda_{\gamma_{i}, [u_{i}]} \) for each \( i = 1, 2, 3 \).
  Define \( \mu \cdot u := u \circ \tau_{\mu} \) to be the left action of
  \( \Gamma \) on the class of finite dimensional unitary representation of
  \( G \), and let \( M(u) \) denote the vector space of matrix coefficients of
  \( u \). Using the character formulae for \( U_{1} \times_{\gamma} U_{2} \)
  and for \( \Psi_{\gamma_{3}}\bigl([(u_{3}, V_{3}, v_{3})]\bigr) \), as well as
  the construction of \( \Psi_{\gamma_{3}} \), we see that as elements in
  \( \pol(G) \otimes C(\Lambda_{\gamma}) \), we have
  \begin{equation}
    \label{eq:713ccdb4315575e8}
    \chi\left(u_{\gamma,\gamma}^{(3)} \vert_{G \rtimes \Lambda_{\gamma}}\right)
    \in \vect\left(\bigcup_{r \in \Lambda} M(r \cdot u_{3})\right)
    \otimes C(\Lambda_{\gamma}) \subseteq \pol(G) \otimes C(\Lambda_{\gamma}),
  \end{equation}
  and
  \begin{equation}
    \label{eq:5fecf32b1b10833f}
    \chi\left(U_{1} \times_{\gamma} U_{2}\right)
    \in \vect\Bigl(\bigl[\Gamma \cdot M(u_{1})\bigr] %
      \bigl[\Gamma \cdot M(u_{2})\bigr]\Bigr)
    \otimes C(\Lambda_{\gamma}),
  \end{equation}
  where
  \begin{displaymath}
    \forall i = 1,2, \qquad
    \Gamma \cdot M(u_{i}) :=
    \bigcup_{r_{i} \in \Gamma} M(r_{i} \cdot u_{i})
  \end{displaymath}
  and \( \bigl[\Gamma \cdot M(u_{1})\bigr] \bigl[\Gamma \cdot M(u_{2})\bigr] \)
  denotes product of form \( \varphi_{1} \varphi_{2} \in \pol(G) \) where
  \( \varphi_{i} \in \Gamma \cdot M(u_{i}) \) for \( i = 1, 2 \). By
  \eqref{eq:713ccdb4315575e8}, \eqref{eq:5fecf32b1b10833f} and a simple
  calculation using the Haar state on
  \( C(G) \otimes C(\Lambda_{\gamma}) = C(G \rtimes \Lambda_{\gamma}) \), it is
  clear that \eqref{eq:1303e2739229a227} implies the existence of
  \( r \in \Lambda \), \( r_{1}, r_{2} \in \Gamma \), such that
  \( M(r \cdot u_{3}) \) and
  \( M(r_{1} \cdot u_{1}) \cdot M(r_{2} \cdot u_{2}) \) are \emph{not}
  orthogonal with respect to the Haar measure on \( G \). Since the
  representation \( r \cdot u_{3} \) of \( G \) is irreducible, this forces that
  \begin{displaymath}
    \dim \morph_{G}\bigl(r \cdot u_{3},
    (r_{1} \cdot u_{1}) \times (r_{2} \cdot u_{2})\bigr) \ne 0.
  \end{displaymath}
  Hence
  \begin{equation}
    \label{eq:748477faf46ef21a}
    r \cdot [u_{3}] \subseteq (r_{1} \cdot [u_{1}]) \otimes (r_{2} \cdot
    [u_{2}]).
  \end{equation}
  Since \( l_{\widehat{G}} \) is a \( \Gamma \)\nobreakdash-invariant length
  function, by \eqref{eq:748477faf46ef21a}, we have
  \begin{equation}
    \label{eq:ee11737e2cae5c84}
    l_{\widehat{G}}([u_{3}]) = l_{\widehat{G}}(r \cdot [u_{3}])
    \leq l_{\widehat{G}}(r_{1} \cdot [u_{1}]) %
    + l_{\widehat{G}}(r_{2} \cdot [u_{2}]) %
    = l_{\widehat{G}}([u_{1}]) + l_{\widehat{G}}([u_{2}]).
  \end{equation}

  On the other hand, \eqref{eq:1303e2739229a227} also implies that
  \( \gamma_{3} \in \mathscr{O}_{3} %
  \subseteq \mathscr{O}_{1} \mathscr{O}_{2} \), so there is
  \( s_{i} \in \mathscr{O}_{i} \), \( i = 1,2 \), such that
  \( s_{1}s_{2} = \gamma_{3} \). Using the fact that \( l_{\Gamma} \) is a
  \( \beta^{\Lambda} \)-invariant length function, we have
  \begin{equation}
    \label{eq:07e48a754bc0c9c4}
    l_{\Gamma}(\gamma_{3}) = l_{\Gamma}(s_{1}s_{2})
    \leq l_{\Gamma}(s_{1}) + l_{\Gamma}(s_{2})
    = l_{\Gamma}(\gamma_{1}) + l_{\Gamma}(\gamma_{2}).
  \end{equation}

  By \eqref{eq:ee11737e2cae5c84}, \eqref{eq:07e48a754bc0c9c4} and
  \eqref{eq:4ec696cc821b4621} again, we have
  \begin{align*}
    l_{\mathscr{O}_{3}}([U_{3}]) %
    &= l_{\widehat{G}}([u_{3}]) + l_{\Gamma}(\gamma_{3}) \\
    &\leq l_{\widehat{G}}([u_{1}]) + l_{\widehat{G}}([u_{2}]) %
      + l_{\Gamma}(\gamma_{1}) + l_{\Gamma}(\gamma_{2}) \\
    &= l_{\mathscr{O}_{1}}([U_{1}]) + l_{\mathscr{O}_{2}}([U_{2}]).
  \end{align*}
  This finishes the proof of the claim, and hence the lemma.
\end{proof}

Now Theorem~\ref{theo:69214233f77b0115} follows from
Lemma~\ref{lemm:57fb1888aba4d546}, Theorem~\ref{theo:ad139c52be8424d9} and
Theorem~\ref{theo:66c73d7c75f29067}.

\begin{rema}
  \label{rema:54472d6e9b8e5557}
  It is however, unknown to the author that whether the polynomial growth
  (resp.\ \( (RD) \)) of the dual
  \( \widehat{\Gamma \bowtie (G \rtimes \Lambda)} \) implies the existence of a
  \emph{\( \Gamma \)-invariant} length function \( l_{\widehat{G}} \) on
  \( \widehat{G} \) witnessing the polynomial growth (resp.\ \( (RD) \)) of
  \( \widehat{G} \). Later we will show that if the composition
  \( \Gamma \xrightarrow{\tau} \aut(G) \to \out(G) \) has finite image, then the
  converse of Theorem~\ref{theo:69214233f77b0115} also holds
  (Theorem~\ref{theo:54211248e165a397}).
\end{rema}

\section{Invariance of length functions and proof of
  Theorem~\ref{theo:54211248e165a397}}
\label{sec:ef1c8a06826abe0a}

In this section, we partially treat the difficulty of the technical assumption
on the \( \Gamma \)-invariance of the length function \( l_{\widehat{G}} \) on
\( \widehat{\Gamma} \) that witnesses the polynomial growth or \( (RD) \) of
\( \widehat{G} \), as presented in Theorem~\ref{theo:69214233f77b0115}. The
results here will be used in \S~\ref{sec:c9560fd86a3f44db} in which we give some
concrete examples of bicrossed products whose dual has \( (RD) \) but does
\emph{not} have polynomial growth. We also point out here that the examples
given in \S~\ref{sec:bd435170e9c7be48} do \emph{not} fit into this framework,
thus we only have a partial understanding of the situation. In particular, the
tools developed here leads to a natural proof of
Theorem~\ref{theo:54211248e165a397}.

We begin by considering a technical lemma on the Fourier transform and the
Sobolev-\( 0 \)-norm in the context of compact quantum groups of Kac type.

\begin{lemm}
  \label{lemm:05429b5c3076ea57}
  Let \( \mathbb{H} \) be a compact quantum group of Kac type. Suppose
  \( \theta : C(\mathbb{H}) \to C(\mathbb{H}) \) is an automorphism of
  \( C^{\ast} \)\nobreakdash-algebras that intertwines the comultiplication
  \( \Delta \) of \( \mathbb{H} \) (i.e.\ \( \theta \) is an automorphism of the
  quantum group \( \mathbb{H} \)). Then there exists an automorphism
  \( \widehat{\theta} \) of the involutive algebra
  \( c_{c}(\widehat{\mathbb{H}}) \), such that
  \begin{equation}
    \label{eq:80c08c00c6248b0a}
    \forall a \in c_{c}(\widehat{\mathbb{H}}), \quad
    \mathcal{F}_{\mathbb{H}}\left(\widehat{\theta}(a)\right) %
    = \theta \bigl(\mathcal{F}_{\mathbb{H}}(a)\bigr)
    \quad \text{ and } \quad %
    \norm*{\widehat{\theta}(a)}_{\mathbb{H},0} = \norm*{a}_{\mathbb{H},0}.
  \end{equation}
\end{lemm}
\begin{proof}
  Choose a complete set of representatives
  \( \set*{u^{x} \given x \in \irr(\mathbb{H})} \) for \( \irr(\mathbb{H}) \),
  and denote the finite dimensional Hilbert space underlying the unitary
  representation \( u^{x} \) by \( \mathscr{H}_{x} \), so that
  \begin{displaymath}
    c_{c}\bigl(\widehat{\mathbb{H}}\bigr) %
    = \bigoplus_{x \in \irr(\mathbb{H})}^{\alg} %
    \mathcal{B}\bigl(\mathscr{H}_{x}\bigr).
  \end{displaymath}
  For each finite dimensional unitary representation
  \( u \in \mathcal{B}(\mathscr{H}) \otimes \pol(\mathbb{H}) \) of
  \( \mathbb{H} \) on \( \mathscr{H} \), since \( \theta \) is an automorphism
  of \( \mathbb{H} \), the unitary operator
  \begin{equation}
    \label{eq:6a6c1e427ca97336}
    \theta_{\ast}(u) := (\id \otimes \theta)(u)
    \in \mathcal{B}(\mathscr{H}) \otimes \pol(\mathbb{H})
  \end{equation}
  remains a unitary representation of \( \mathbb{H} \) on the same space
  \( \mathscr{H} \). It is clear that \( \theta_{\ast} \) also passes to a
  bijection of the set \( \irr(\mathbb{H}) \) to itself, which we still denote
  by \( \theta_{\ast} \) by abuse of notation, via
  \( \theta_{\ast}([u]) = [\theta_{\ast}(u)] \). In particular, for each
  \( x \in \irr(\mathbb{H}) \), we have
  \( \left[u^{\theta_{\ast}(x)}\right] = \theta_{\ast}(x) =
  \left[\theta_{\ast}(u^{x})\right] \), thus there exists a \emph{unitary}
  \begin{displaymath}
    T_{x} \in \morph_{\mathbb{H}} \left( %
      u^{\theta_{\ast}(x)}, \theta_{\ast}(u^{x})\right) %
    \subseteq \mathcal{B}\left(\mathscr{H}_{\theta_{\ast}(x)}, %
      \mathscr{H}_{x}\right),
  \end{displaymath}
  which is uniquely determined up to a scalar in \( \mathbb{T} \).

  Take any
  \begin{equation}
    \label{eq:132832ce7b63ae8a}
    a = {(a_{x})}_{x \in \irr(\mathbb{H})}
    = \sum_{x \in \irr(\mathbb{H})} a_{x}
    \in c_{c}(\widehat{\mathbb{H}}),
  \end{equation}
  where the sum is finite (meaning all but finitely many
  \( a_{x} \in \mathcal{B}(\mathscr{H}_{x}) \) is \( 0 \)). For each
  \( x \in \irr(\mathbb{H}) \), we set
  \begin{equation}
    \label{eq:404f4e5c5d1b9686}
    b_{\theta_{\ast}(x)} := T_{x}^{\ast} a_{x} T_{x}
    \in \mathcal{B}\left(\mathscr{H}_{\theta_{\ast}(x)}\right).
  \end{equation}
  Then
  \begin{equation}
    \label{eq:afef6dfac1d7f171}
    \dim\bigl(\theta_{\ast}(x)\bigr) = \dim x.
  \end{equation}
  By the choice of \( T_{x} \), we have
  \begin{equation}
    \label{eq:4f6a7f0376f937b9}
    \begin{split}
      & \leadmathskip %
      \left( \tr_{\mathscr{H}_{\theta_{\ast}(x)}} \otimes \id \right) %
      \Bigl( u^{\theta_{\ast}(x)}(b_{\theta_{\ast}(x)} \otimes 1) \Bigr) \\
      &= \left( \tr_{\mathscr{H}_{\theta_{\ast}(x)}} \otimes \id \right) %
      \Bigl( %
      (T_{x}^{\ast} \otimes 1) %
      \bigl[ [\theta_{\ast}(u^{x})](a_{x} \otimes 1)\bigr] %
      (T_{x} \otimes 1) %
      \Bigr) \\
      &= \left( \tr_{\mathscr{H}_{x}} \otimes \id \right) %
      \Bigl( [\theta_{\ast}(u^{x})](a_{x} \otimes 1) \Bigr) \\
      &= \theta \left[ %
        \left( \tr_{\mathscr{H}_{x}} \otimes \id \right) %
        \Bigl( [u^{x}(a_{x} \otimes 1) \Bigr) %
      \right],
    \end{split}
  \end{equation}
  and
  \begin{equation}
    \label{eq:b478a88896cd4329}
    \begin{split}
      & \leadmathskip %
      \left( \tr_{\mathscr{H}_{\theta_{\ast}(x)}} \otimes \id \right) %
      \bigl( b_{\theta_{\ast}(x)}^{\ast}b_{\theta_{\ast}(x)} \bigr) \\
      &= \left( \tr_{\mathscr{H}_{\theta_{\ast}(x)}} \otimes \id \right) %
      \bigl( T_{x}^{\ast} a_{x}^{\ast}a_{x}T_{x} \bigr) \\
      &= \left( \tr_{\mathscr{H}_{\theta_{\ast}(x)}} \otimes \id \right) %
      (a_{x}^{\ast}a_{x}).
    \end{split}
  \end{equation}
  We now define
  \begin{equation}
    \label{eq:8a14a807a0796f3a}
    \widehat{\theta}(a) := \sum_{x \in \irr(\mathbb{H})} b_{\theta_{\ast}(x)}
  \end{equation}
  Since \( \theta_{\ast} : \irr(\mathbb{H}) \to \irr(\mathbb{H}) \) is a
  bijection, it is clear that \eqref{eq:8a14a807a0796f3a} defines an
  automorphism \( \widehat{\theta} \) of the involutive algebra
  \( c_{c}(\widehat{\mathbb{H}}) \). Finally, \eqref{eq:80c08c00c6248b0a}
  follows from \eqref{eq:4f6a7f0376f937b9} and \eqref{eq:b478a88896cd4329}.
\end{proof}
\begin{rema}
  \label{rema:c92276a65161534d}
  Lemma~\ref{lemm:05429b5c3076ea57} also applies to non-Kac type
  \( \mathbb{H} \) with almost the same proof, with the caveats that the Fourier
  transform and the Sobolev norms need to be adjusted using quantum dimensions
  of representations, which is not needed for our purposes hence not introduced
  here (see \cite{MR2329000} or \cite{MR3448333} the discussion of non-Kac type
  Fourier transforms and Sobolev norms).
\end{rema}

\begin{prop}
  \label{prop:461c0747dedf7cfc}
  Let \( \mathbb{H} = \bigl(C(\mathbb{H}), \Delta\bigr) \) be a compact quantum
  group of Kac type. Suppose \( \Theta \) is a \emph{finite} subgroup of
  \( \aut\bigl(C(\mathbb{H}), \Delta\bigr) \). The following are equivalent.
  \begin{enumerate}
  \item \label{item:d69e6bf8551d304b} There exists a length function \( l \) on
    \( \widehat{\mathbb{H}} \) and \( P(X) \in \mathbb{R}[X] \), such that
    \begin{equation}
      \label{eq:b07f86faf4f8f4ca}
      \forall k \in \mathbb{N}, \quad a \in Q_{l,k} c_{c}(\widehat{\mathbb{H}})
      \implies \norm*{\mathcal{F}_{\mathbb{H}}(a)}
      \leq P(k) \norm*{a}_{\mathbb{H},0},
    \end{equation}
    where
    \( Q_{l,n}:= \sum_{x \in \irr(\mathbb{H},\, l(x) < n+1)}p_{x} \in
    \ell^{\infty}(\widehat{\mathbb{H}}) \).
  \item \label{item:9784581b8a95616b} There exists a
    \emph{\( \Theta \)-invariant} length function \( l_{\Theta} \) on
    \( \widehat{\mathbb{H}} \) and \( Q(X) \in \mathbb{R}[X] \), such that
    \begin{equation}
      \label{eq:0cf9d8dd9d8615c1}
      \forall k \in \mathbb{N}, \quad %
      a \in Q_{l_{\Theta},k} c_{c}(\widehat{\mathbb{H}})
      \implies \norm*{\mathcal{F}_{\mathbb{H}}(a)}
      \leq Q(k) \norm*{a}_{\mathbb{H},0}.
    \end{equation}
  \end{enumerate}
\end{prop}
\begin{proof}
  Obviously \ref{item:9784581b8a95616b} implies \ref{item:d69e6bf8551d304b}.

  Now suppose \ref{item:d69e6bf8551d304b} holds and let's prove
  \ref{item:9784581b8a95616b}. Let \( n = \abs*{\Theta} \) and suppose
  \( \theta_{1}, \ldots, \theta_{n} \) form an enumeration of all elements of
  \( \Theta \). Let \( l_{i} \) denote the length function
  \( l \circ {(\theta_{i})}_{\ast} \) on \( \widehat{\mathbb{H}} \) (see the
  discussion after \eqref{eq:6a6c1e427ca97336} in the proof of
  Lemma~\ref{lemm:05429b5c3076ea57}). Put
  \begin{equation}
    \label{eq:c49cb9a7f08b2823}
    l_{\Theta} := \frac{1}{\abs*{\Theta}} \sum_{i=1}^{n} l_{i},
  \end{equation}
  then it is clear that \( l_{\Theta} \) is a \( \Theta \)-invariant length
  function on \( \widehat{\mathbb{H}} \). For each \( k \in \mathbb{N} \),
  define
  \begin{equation}
    \label{eq:37a383a675accfa9}
    F_{\Theta, k} := \set*{x \in \irr(\mathbb{H}) \given l_{\Theta}(x) < k+1},
  \end{equation}
  and for \( i = 1, \ldots, n \), put
  \begin{equation}
    \label{eq:59d4796a788393f8}
    F_{i, k}  := \set*{x \in \irr(\mathbb{H}) \given
      l_{\theta_{i}}(x) < k+1},
  \end{equation}
  By \eqref{eq:b07f86faf4f8f4ca}, \eqref{eq:59d4796a788393f8} and
  \eqref{eq:c49cb9a7f08b2823}, we have
  \begin{equation}
    \label{eq:59b5584636a357d8}
    F_{\Theta,k} \subseteq \bigcup_{i=1}^{n} F_{i, k}.
  \end{equation}
  Define
  \begin{equation}
    \label{eq:1f93cf679787adc4}
    \begin{split}
      \xi : F_{\Theta,k} & \to \set*{1, \ldots, n} \\
      x & \mapsto \inf\set*{i \given x \in F_{i,k}}.
    \end{split}
  \end{equation}
  Note that \eqref{eq:59b5584636a357d8} guarantees that \( \xi \) is
  well-defined.

  We now prove \eqref{eq:0cf9d8dd9d8615c1} holds for some suitable polynomial
  \( Q(X) \in \mathbb{R}[X] \), which will finish the proof.  Since
  \( a \in Q_{l_{\Theta},k} c_{c}(\widehat{\mathbb{H}}) \), there exists a
  finite subset \( F \) of \( F_{\Theta, k} \), such that
  \begin{equation}
    \label{eq:4dfd63ec4cce61b0}
    a = \sum_{x \in F} a_{x} = \sum_{i=1}^{n} a_{i},
  \end{equation}
  where for each \( i \),
  \begin{equation}
    \label{eq:b6658e966b3555e9}
    a_{i} := \sum_{x \in F \cap \xi^{-1}(i)} a_{x} %
    \in Q_{l_{i},k}.
  \end{equation}
  By Lemma~\ref{lemm:05429b5c3076ea57} and \ref{item:d69e6bf8551d304b}, we have
  \begin{equation}
    \label{eq:f9682190ba1bda19}
    \forall i = 1, \ldots, n, \quad
    \norm*{\mathcal{F}_{\mathbb{H}}(a_{i})}
    \leq P(k) \norm*{a_{i}}_{\mathbb{H},0},
  \end{equation}
  hence
  \begin{equation}
    \label{eq:0c41ec99d9f7ac79}
    \begin{split}
      \norm*{\mathcal{F}_{\mathbb{H}}(a)}^{2} %
      &\leq {\left(\sum_{i=1}^{n} %
          \norm*{\mathcal{F}_{\mathbb{H}}(a_{i})} \right)}^{2} %
      \leq n \left(\sum_{i=1}^{n} \norm*{\mathcal{F}_{\mathbb{H}}(a_{i})}^{2} %
      \right) \\
      &\leq n {[P(k)]}^{2} \left( %
        \sum_{i=1}^{n}\norm*{a_{i}}_{\mathbb{H},0}^{2} \right)
      = \abs*{\Theta} {[P(k)]}^{2} \norm*{a}_{\mathbb{H},0}^{2}.
    \end{split}
  \end{equation}
  Thus posing \( Q(X) = \sqrt{\abs*{\Theta}} P(X) \in \mathbb{R}[X] \), we
  have~\eqref{eq:0cf9d8dd9d8615c1}.
\end{proof}

\begin{coro}
  \label{coro:f15d5d55ec591c2a}
  The following are equivalent:
  \begin{enumerate}
  \item \label{item:d4b5b518a3d64ece} \( \Gamma \) has polynomial growth (resp.\
    \( (RD) \));
  \item \label{item:900be9aafe129e4d} there exists a
    \( \beta^{\Lambda} \)-invariant length function \( l_{\Gamma} \) on
    \( \Gamma \), such that \( \pair*{\Gamma}{l_{\Gamma}} \) has polynomial
    growth (resp.\ \( (RD) \)).
  \end{enumerate}
\end{coro}
\begin{proof}
  This follows from Proposition~\ref{prop:461c0747dedf7cfc} by posing
  \( \Theta = \set*{\adj_{r} \in \aut(\Gamma) \given r \in \Lambda} \) and
  \( \mathbb{H} = \widehat{\Gamma} \).
\end{proof}

\begin{proof}[Proof of Theorem~\ref{theo:54211248e165a397}]
  We begin by observing more closely the action
  \( \Gamma \curvearrowright \irr(G) \). It is clear that this action is
  actually given by \( \aut(G) \) acting on \( G \), and the group morphism
  \( \tau : \Gamma \to \aut(G) \) with respect to which we form the semidirect
  product (see the beginning of \S~\ref{sec:e9d9394edcf4dfb6}). More precisely,
  there is a natural action \( \aut(G) \curvearrowright \irr(G) \) by letting
  \( (\theta, [u]) \mapsto [\theta_{\ast}(u)] \), and the action
  \( \Gamma \curvearrowright \irr(G) \) is given by
  \( (\gamma, x) \mapsto \tau(\gamma) \cdot x \). By definition, one has
  \begin{equation}
    \label{eq:70a011cff261a6d7}
    \inn(G) \subseteq \bigcap_{x \in \irr(G)} {[\aut(G)]}_{x},
  \end{equation}
  where
  \begin{displaymath}
    {[\aut(G)]}_{x} := \set*{\theta \in \aut(G) \given \theta \cdot x = x}.
  \end{displaymath}
  Thus passing to the quotient, it is in fact \( \out(G) = \aut(G)/\inn(G) \)
  that acts on \( \irr(G) \) (inner automorphisms fixes equivalence classes of
  representations by definition). Thus to talk about the \( \Gamma \) invariance
  of a given length function \( l \) on \( \widehat{G} \), it suffices to
  consider the invariance of \( l \) under the image of the composition of group
  morphisms \( \tau \colon \Gamma \to \aut(G) \) and the canonical projection
  \( \aut(G) \to \out(G) \).

  With the above considerations in mind, Theorem~\ref{theo:54211248e165a397} now
  follows from Theorem~\ref{theo:69214233f77b0115},
  Proposition~\ref{prop:461c0747dedf7cfc} (posing \( \mathbb{H} = G \) and
  \( \Theta = \image(\widetilde{\tau}) \)) and
  Corollary~\ref{coro:f15d5d55ec591c2a}.
\end{proof}

\section{Examples of bicrossed products with rapid decay but not polynomial
  growth--part I}
\label{sec:c9560fd86a3f44db}

In this section, we construct some examples of bicrossed products whose has
property \RD but does not have polynomial growth, using the framework developed
in \S~\ref{sec:ef1c8a06826abe0a}. We shall frequently use Jolissaint's theorem
on rapid decay of amalgamated product of groups, which we record here for
convenience of the reader.

\begin{theo}[Jolissaint]
  \label{theo:b5941076363c875f}
  Suppose \( \Gamma_{1} \), \( \Gamma_{2} \) are two discrete groups with
  property \( (RD) \), \( A \) is a finite group,
  \( j_{i} : A \hookrightarrow \Gamma_{i} \) is an injective group morphism for
  \( i = 1, 2 \), then the amalgamated product
  \( \Gamma_{1} \ast_{A} \Gamma_{2} \) with respect to \( j_{1} \), \( j_{2} \)
  also has property \( (RD) \).
\end{theo}
\begin{proof}
  This is part of \cite{MR943303}*{Theorem 2.2.2}.
\end{proof}

We will refer Theorem~\ref{theo:b5941076363c875f} as Jolissaint's theorem
hereafter.

\begin{exam}
  \label{exam:5b12fb782ad37377}
  Take
  \( \Gamma = \psl_{2}(Z) \simeq (\mathbb{Z} / 2\mathbb{Z}) \ast (\mathbb{Z} /
  3\mathbb{Z}) \), with the isomorphism determined by identifying
  \( \mathbb{Z} / 2\mathbb{Z} \) with the cyclic group generated by
  \( s \in \Gamma \), and \( \mathbb{Z} / 3 \mathbb{Z} \) with the cyclic group
  generated by \( t \in \Gamma \), where
  \begin{displaymath}
    s =
    \begin{pmatrix}
      0 & 1 \\
      -1 & 0
    \end{pmatrix}
    \quad\text{ and }\quad
    t =
    \begin{pmatrix}
      0 & -1 \\
      1 & 1
    \end{pmatrix}.
  \end{displaymath}
  (see e.g.~\cite{MR2391387}*{Example E.10, page 476} for a discussion of this
  amalgamated product decomposition of \( \speciallinear_{2}(\mathbb{Z}) \)).
  Let \( G \) be any compact connected real Lie group that admits an element
  \( x \in G \) of order \( 2 \), and an element \( y \) of order \( 3 \), such
  that \( \set*{x,y} \not\subseteq Z(G) \) (e.g.\
  \( G = \operatorname{SO}(3, \mathbb{R}) \), \( x \) is any rotation by
  \( \pi \), \( y \) any rotation by \( 2 \pi / 3 \)), where \( Z(G) \) is the
  center of \( G \). Now the mapping \( s \mapsto \adj_{x} \),
  \( t \mapsto \adj_{y} \) determines a unique group morphism
  \begin{displaymath}
    \tau : \Gamma \to \inn(G) \subseteq \aut(G)
  \end{displaymath}
  so \( \widetilde{\tau} : \Gamma \to \out(G) \) is trivial (hence of finite
  image).  Put \( \Lambda < \Gamma \) to be \( \langle s \rangle \) or
  \( \langle t \rangle \) (the subgroup generated by \( s \) or \( t \)
  respectively). Since \( \Lambda \not\subseteq Z(\Gamma) \), it follows from
  the choice of \( x \) and \( y \) that the resulted bicrossed product
  \( \mathbb{G}:= \Gamma {}_{\alpha^{\Lambda}}\bowtie_{\beta^{\Lambda}} (G
  \rtimes_{\tau} \Lambda) \) is nontrivial
  (Proposition~\ref{prop:71acce0cf99cd889}).

  By Jolissaint's theorem, \( \psl_{2}(Z) \) has \( (RD) \), but
  \( \psl_{2}(\mathbb{Z}) \) does not have polynomial growth since it is not
  virtually nilpotent (by Gromov's famous theorem on polynomial growth,
  see~\cite{MR623534}), and \cite{MR2329000} showed that \( \widehat{G} \) has
  polynomial growth, thus Theorem~\ref{theo:54211248e165a397} applies and we see
  that \( \widehat{\mathbb{G}} \) has \( (RD) \) but not polynomial growth.
\end{exam}

\begin{exam}
  Let \( G \) be any compact group with \( \widehat{G} \) having polynomial
  growth (e.g.\ all connected compact real Lie group), and \( \Lambda \) a
  finite subgroup of \( \aut(G) \). Take \( \Gamma \) to be a nontrivial
  semidirect product of the free group \( \mathbb{F}_{2} \) on two generators
  (here \( \mathbb{F}_{2} \) can be replaced by any discrete group with
  \( (RD) \) but without polynomial growth) with \( \Lambda \) (in particular,
  \( \Lambda \) is nontrivial). Then the obvious action of \( \Lambda \) on
  \( G \) and the canonical projection
  \( \mathbb{F}_{2} \rtimes \Lambda \to \Lambda \) together yield a nontrivial
  left action \( \tau \) of \( \Gamma \) on \( G \) by topological
  automorphisms. The same reasoning as in Example \ref{exam:5b12fb782ad37377}
  shows that \( \Gamma \bowtie (G \rtimes \Lambda) \) is also a bicrossed
  product whose dual has \( (RD) \) but not polynomial growth.
\end{exam}

Many more examples can be constructed in the same spirit as in the above
examples, showing that Theorem~\ref{theo:54211248e165a397} is an applicable
procedure to produce bicrossed products whose dual has \( (RD) \) but not
polynomial growth.

\section{Examples of bicrossed products with rapid decay but not polynomial
  growth--part II}
\label{sec:bd435170e9c7be48}

Despite of the fact that Theorem~\ref{theo:54211248e165a397} yields many
interesting concrete examples of bicrossed products with property \( (RD) \) as
shown in \S~\ref{sec:c9560fd86a3f44db}, it is worth pointing out that the
restriction the finiteness of the image \( \image(\widetilde{\tau}) \) is too
strong to include many other interesting examples, which we will now show in
this section. To make the contrast even more dramatic, we show how to construct
examples of nontrivial bicrossed product of the form
\( \Gamma \bowtie (G \rtimes \Lambda) \) whose dual has \( (RD) \) but
\emph{not} polynomial growth, while \( \image(\widetilde{\tau}) \) as in
Theorem~\ref{theo:54211248e165a397} is infinite (hence
Theorem~\ref{theo:54211248e165a397} no longer applies).

We begin with a simple result in finite group theory.

\begin{lemm}
  \label{lemm:a070dd629af6e21d}
  If \( A \) is a finite abelian group, then there exists infinitely many finite
  abelian groups \( B \), such that \( A \) is isomorphic to a subgroup of
  \( \aut(B) \).
\end{lemm}
\begin{proof}
  Since \( A \) is a direct sum of finite cyclic groups, without loss of
  generality, we may assume \( A \) is cyclic of order \( n \), with \( a \) as
  a generator. Set \( B \) to be the \( n \)-fold direct sum of any nontrivial
  finite abelian group \( C \), and define \( \sigma(a) \in \aut(B) \) to be the
  permutation
  \begin{displaymath}
    (c_{1}, \ldots, c_{n}) \mapsto (c_{2}, \ldots, c_{n}, c_{1}).
  \end{displaymath}
  Then it is clear that
  \begin{align*}
    \sigma : A &\to \aut(B) \\
    a^{m} &\mapsto {[\sigma(a)]}^{m}
  \end{align*}
  is a well-defined injective group morphism.
\end{proof}

As we will see later, Theorem~\ref{theo:54211248e165a397} no longer applies for
the examples constructed in this section due to the violation of the hypothesis
of the finiteness of \( \image(\widetilde{\tau}) \).  We will thus have to
resort to Theorem~\ref{theo:69214233f77b0115} to prove the rapid decay of the
dual of the bicrossed product \( \Gamma \bowtie (G \rtimes \Lambda) \). Here,
the \( \beta^{\Lambda} \)-invariance of the length function on \( \Gamma \)
poses no problem thanks to Corollary~\ref{coro:f15d5d55ec591c2a}. But the
\( \Gamma \)-invariance of the length function on \( \widehat{G} \) requires a
little more work.

\begin{lemm}
  \label{lemm:146b6e7b1d1a615a}
  Suppose \( \Xi_{1}, \Xi_{2}, \ldots \) is a sequence of finite discrete (hence
  compact) groups. The product group \( \prod_{i=1}^{\infty} \aut(\Xi_{i}) \)
  naturally acts pointwise on the direct sum
  \( \oplus_{i=1}^{\infty} \Xi_{i} \), hence we have a canonical inclusion
  \( \prod_{i=1}^{\infty} \aut(\Xi_{i}) \subseteq
  \aut(\oplus_{i=1}^{\infty}\Xi_{i}) \). With these settings, there exists a
  \( \prod_{i=1}^{\infty}\aut(\Xi_{i}) \)-invariant length function \( l \) on
  the discrete group \( \oplus_{i=1}^{\infty}\Xi_{i} \), such that the pair
  \( \pair*{\oplus_{i=1}^{\infty} \Xi_{i}}{l} \) has polynomial growth.
\end{lemm}
\begin{proof}
  Let \( N_{i} = \abs*{\Xi_{i}} \) for all \( i \in \mathbb{N}_{>0} \) and pose
  \( M_{k} = \prod_{i=1}^{k} N_{i} \) for all \( k \in \mathbb{N} \) (we make
  the convention that \( M_{0} = 1 \)). Let \( e_{i} \) be the identity of the
  group \( \Xi_{i} \), and denote the characteristic function of
  \( \Xi_{i} \setminus \set*{e_{i}} \) by \( \chi_{i} \). Define
  \begin{align*}
    l : \oplus_{i=1}^{\infty} \Xi_{i} & \to \mathbb{R}_{\geq0} \\
    (\xi_{i}) & \mapsto \sum_{i=1}^{\infty} \chi_{i}(\xi_{i}) M_{i}.
  \end{align*}
  Then it is clear that \( l \) is a
  \( \prod_{i=1}^{\infty}\aut(G_{i}) \)-invariant length function on
  \( \oplus_{i=1}^{\infty} \Xi_{i} \). Moreover, for all
  \( n \in \mathbb{N}_{>0} \), there exists a unique \( k \geq 1 \), such that
  \( M_{k-1} \leq n < M_{k} \). Then, by the definition of \( l \), we have
  \begin{displaymath}
    \set*{\xi = (\xi_{i}) \in \oplus_{i=1}^{\infty} \Xi_{i} \given l(\xi) < n}
    \subseteq \set*{(\xi_{i}) \in \oplus_{i=1}^{\infty} \Xi_{i} \given \forall i \geq k, \, \xi_{i} = e_{i}}.
  \end{displaymath}
  Thus
  \begin{displaymath}
    \abs*{\set*{\xi = (\xi_{i}) \in \oplus_{i=1}^{\infty} \Xi_{i} \given l(\xi) < n}}
    \leq \prod_{i=1}^{k-1} N_{i} = M_{k-1} \leq n.
  \end{displaymath}
  In particular, \( \pair*{\oplus_{i=1}^{\infty} \Xi_{i}}{l} \) has polynomial growth.
\end{proof}

We are now prepared to give the construction of new examples of bicrossed
product of the form \( \Gamma \bowtie (G \rtimes \Lambda) \) that don't fit into
the framework of Theorem~\ref{theo:54211248e165a397}.

\begin{exam}
  Let \( \Lambda \) be any nontrivial finite abelian group. By
  Lemma~\ref{lemm:a070dd629af6e21d}, one can take a sequence of finite abelian
  groups \( {(G_{i})}_{i=1}^{\infty} \), such that \( \Lambda \) is isomorphic
  to a subgroup of \( \aut(G_{i}) \) for each \( i = 1,2,\ldots \) via an
  injective group morphism \( j_{i} : \Lambda \hookrightarrow \aut(G_{i}) \).
  Equip each \( G_{i} \) with the discrete topology, and
  \( G := \prod_{i=1}^{\infty} G_{i} \) the product topology. Then \( G \) is a
  compact abelian group. In particular, the character group \( \chi(G) \) of
  \( G \) is a complete set of representatives of \( \irr(G) \). By Pontryagin's
  duality, we have \( \chi(G) \simeq \oplus_{i=1}^{\infty} \chi(G_{i}) \), and
  it is clear that length functions on \( \widehat{G} \) become exactly length
  functions on the discrete group \( \chi(G) \) of continuous characters of
  \( G \). But as finite abelian groups, each \( G_{i} \) is isomorphic to
  \( \chi(G_{i}) \) (albeit the isomorphism is not natural in the categorical
  sense). Thus Lemma~\ref{lemm:146b6e7b1d1a615a} shows that there exists a
  \( \prod_{i=1}^{\infty} \aut(G_{i}) \)-invariant length function \( l_{G} \)
  on \( \widehat{G} \), such that \( \pair*{G}{l_{G}} \) has polynomial growth,
  where we've used the canonical inclusion
  \( \prod_{i=1}^{\infty} \aut(G_{i}) \subseteq \aut(G) \).

  The construction of \( \Gamma \) takes some more work which we now
  explain. First we take \( \Lambda' \) to be any nontrivial finite group and
  pose \( \Gamma_{1} \) to be the free product \( \Lambda \ast \Lambda' \). It
  follows from Jolissaint's theorem and Gromov's theorem that \( \Gamma_{1} \)
  has \( (RD) \) but not polynomial growth. Define
  \( j : \Lambda \hookrightarrow \prod_{i=1}^{\infty} \aut(G_{i}) \) to be the
  mapping
  \( \lambda \mapsto \bigl(j_{1}(\lambda), j_{2}(\lambda), \cdots\bigr) \).
  Take any infinite discrete subgroup \( \Gamma_{2}' \) of
  \( \oplus_{i=1}^{\infty}\aut(G_{i}) \subseteq \prod_{i=1}^{\infty} \aut(G_{i})
  \) such that \( j(\Lambda) \) is contained in the normalizer of
  \( \Gamma_{2}' \) in \( \prod_{i=1}^{\infty} \aut(G_{i}) \). Obviously
  \( j(\Lambda) \) and \( \Gamma_{2}' \) intersect trivially, thus the subgroup
  of \( \prod_{i=1}^{\infty} \aut(G_{i}) \) generated by \( j(\Lambda) \) and
  \( \Gamma_{2}' \) is the (internal) semidirect product of \( \Gamma_{2}' \)
  with \( j(\Lambda) \), which we denote by \( \Gamma_{2} \). Since
  \( \oplus_{i=1}^{\infty} \aut(G_{i}) \) has polynomial growth by
  Lemma~\ref{lemm:146b6e7b1d1a615a}, it follows that \( \Gamma_{2}' \), hence
  \( \Gamma_{2} \) (note that \( [\Gamma_{2} : \Gamma_{2}'] = \abs*{\Lambda} \)
  is finite) has polynomial growth. In particular, \( \Gamma_{2} \) has
  \( (RD) \), and
  \( j : \Lambda \hookrightarrow \prod_{i=1}^{\infty} \aut(G_{i}) \) restricts
  an injective group morphism, which we still denote by \( j \), from
  \( \Lambda \) into \( \Gamma_{2} \). To facilitate our discussion, we identify
  \( \Lambda \) with its copy in \( \Gamma_{1} = \Lambda \ast \Lambda' \) and in
  \( \Gamma \) via \( j \). This allows us to form the amalgamated product of
  \( \Gamma_{1} \) and \( \Gamma_{2} \) over \( A \), which we denote by
  \( \Gamma \). Jolissaint's theorem applies again and proves that \( \Gamma \)
  has \( (RD) \). Moreover, \( \Gamma \) does not have polynomial growth since
  its subgroup \( \Gamma_{1} \) does not. We also make the obvious
  identification of \( \Lambda \) with \( j(\Lambda) \) in \( \Gamma \). By
  Corollary~\ref{coro:f15d5d55ec591c2a}, there exists a \( \Lambda \)-invariant
  length function \( l_{\Gamma} \) on \( \Gamma \), meaning
  \( l_{\Gamma} = l_{\Gamma} \circ \adj_{r} \) for all \( r \in \Lambda \), such
  that \( \pair*{\Gamma}{l_{\Gamma}} \) has \( (RD) \).

  Finally, let's explain how the action, which is a group morphism
  \( \tau : \Gamma \to \aut(G) \), is defined. The trivial group morphism
  \( \Lambda' \to \aut(G) \), together with
  \( j : \Lambda \to \prod_{i=1}^{\infty} \aut(G_{i}) \subseteq \aut(G) \) and
  the universal property of free products, yields a group morphism
  \( \tau_{1} : \Gamma_{1} \to \aut(G) \).  Let \( \tau_{2} \) be the simple
  inclusion
  \( \Gamma_{2} \hookrightarrow \prod_{i=1}^{\infty}\aut(G_{i}) \subseteq
  \aut(G) \). It is clear that \( \tau_{1} \) and \( \tau_{2} \) agree on
  \( \Lambda \), thus the universal property of
  \( \Gamma_{1} \ast_{\Lambda} \Gamma_{2} \) applies and determines a unique
  group morphism \( \tau : \Gamma \to \aut(G) \). We can finally construct the
  bicrossed product \( \Gamma \bowtie (G \rtimes \Lambda) \), and we conclude by
  Theorem~\ref{theo:69214233f77b0115} that the dual of
  \( \Gamma \bowtie (G \rtimes \Lambda) \) has \( (RD) \) (it does not have
  polynomial growth because of Theorem~\ref{theo:ad139c52be8424d9} and the fact
  that \( \Gamma \) does not have polynomial growth).

  It is clear by our construction that
  \( \image(\tau) = \Gamma_{2} \subseteq \aut(G) = \out(G) \) is infinite, thus
  Theorem~\ref{theo:54211248e165a397} does not apply.
\end{exam}

\section*{Declarations}
\subsection*{Ethical Approval}
Not applicable

\subsection*{Funding}
This work was supported by the French ANR ANCG project No. ANR-19-CE40-0002, and
the Polish National Science Center grant (NCN) 2020/39/I/ST1/01566.

\subsection*{Availability of data and materials}
Not applicable



\begin{bibdiv}
\begin{biblist}

\bib{MR3448333}{article}{
      author={Bhowmick, Jyotishman},
      author={Voigt, Christian},
      author={Zacharias, Joachim},
       title={Compact quantum metric spaces from quantum groups of rapid
  decay},
        date={2015},
        ISSN={1661-6952},
     journal={J. Noncommut. Geom.},
      volume={9},
      number={4},
       pages={1175\ndash 1200},
         url={https://doi-org.scd1.univ-fcomte.fr/10.4171/JNCG/220},
      review={\MR{3448333}},
}

\bib{MR2391387}{book}{
      author={Brown, Nathanial~P.},
      author={Ozawa, Narutaka},
       title={{$C^*$}-algebras and finite-dimensional approximations},
      series={Graduate Studies in Mathematics},
   publisher={American Mathematical Society, Providence, RI},
        date={2008},
      volume={88},
        ISBN={978-0-8218-4381-9; 0-8218-4381-8},
         url={https://doi.org/10.1090/gsm/088},
      review={\MR{2391387}},
}

\bib{MR3299063}{article}{
      author={Cheng, Chuangxun},
       title={A character theory for projective representations of finite
  groups},
        date={2015},
        ISSN={0024-3795},
     journal={Linear Algebra Appl.},
      volume={469},
       pages={230\ndash 242},
         url={https://doi.org/10.1016/j.laa.2014.11.027},
      review={\MR{3299063}},
}

\bib{MR3743231}{article}{
      author={Fima, Pierre},
      author={Mukherjee, Kunal},
      author={Patri, Issan},
       title={On compact bicrossed products},
        date={2017},
        ISSN={1661-6952},
     journal={J. Noncommut. Geom.},
      volume={11},
      number={4},
       pages={1521\ndash 1591},
         url={https://doi.org/10.4171/JNCG/11-4-10},
      review={\MR{3743231}},
}

\bib{MR4345210}{article}{
      author={Fima, Pierre},
      author={Wang, Hua},
       title={Rapid decay and polynomial growth for bicrossed products},
        date={2021},
        ISSN={1661-6952},
     journal={J. Noncommut. Geom.},
      volume={15},
      number={3},
       pages={1105\ndash 1128},
         url={https://doi-org.scd1.univ-fcomte.fr/10.4171/jncg/433},
      review={\MR{4345210}},
}

\bib{MR623534}{article}{
      author={Gromov, Mikhael},
       title={Groups of polynomial growth and expanding maps},
        date={1981},
        ISSN={0073-8301},
     journal={Inst. Hautes \'{E}tudes Sci. Publ. Math.},
      number={53},
       pages={53\ndash 73},
         url={http://www.numdam.org/item?id=PMIHES_1981__53__53_0},
      review={\MR{623534}},
}

\bib{MR943303}{article}{
      author={Jolissaint, Paul},
       title={Rapidly decreasing functions in reduced {$C^*$}-algebras of
  groups},
        date={1990},
        ISSN={0002-9947},
     journal={Trans. Amer. Math. Soc.},
      volume={317},
      number={1},
       pages={167\ndash 196},
         url={https://doi-org.scd1.univ-fcomte.fr/10.2307/2001458},
      review={\MR{943303}},
}

\bib{MR0229061}{article}{
      author={Kac, G.~I.},
       title={Group extensions which are ring groups},
        date={1968},
     journal={Mat. Sb. (N.S.)},
      volume={76 (118)},
       pages={473\ndash 496},
      review={\MR{0229061}},
}

\bib{MR0031489}{article}{
      author={Mackey, George~W.},
       title={Imprimitivity for representations of locally compact groups.
  {I}},
        date={1949},
        ISSN={0027-8424},
     journal={Proc. Nat. Acad. Sci. U. S. A.},
      volume={35},
       pages={537\ndash 545},
         url={https://doi.org/10.1073/pnas.35.9.537},
      review={\MR{0031489}},
}

\bib{MR0098328}{article}{
      author={Mackey, George~W.},
       title={Unitary representations of group extensions. {I}},
        date={1958},
        ISSN={0001-5962},
     journal={Acta Math.},
      volume={99},
       pages={265\ndash 311},
         url={https://doi.org/10.1007/BF02392428},
      review={\MR{0098328}},
}

\bib{MR3204665}{book}{
      author={Neshveyev, Sergey},
      author={Tuset, Lars},
       title={Compact quantum groups and their representation categories},
      series={Cours Sp\'{e}cialis\'{e}s [Specialized Courses]},
   publisher={Soci\'{e}t\'{e} Math\'{e}matique de France, Paris},
        date={2013},
      volume={20},
        ISBN={978-2-85629-777-3},
      review={\MR{3204665}},
}

\bib{MR1970242}{article}{
      author={Vaes, Stefaan},
      author={Vainerman, Leonid},
       title={Extensions of locally compact quantum groups and the bicrossed
  product construction},
        date={2003},
        ISSN={0001-8708},
     journal={Adv. Math.},
      volume={175},
      number={1},
       pages={1\ndash 101},
  url={https://doi-org.scd1.univ-fcomte.fr/10.1016/S0001-8708(02)00040-3},
      review={\MR{1970242}},
}

\bib{MR2329000}{article}{
      author={Vergnioux, Roland},
       title={The property of rapid decay for discrete quantum groups},
        date={2007},
        ISSN={0379-4024},
     journal={J. Operator Theory},
      volume={57},
      number={2},
       pages={303\ndash 324},
      review={\MR{2329000}},
}

\bib{wang2019representations}{article}{
      author={Wang, Hua},
       title={On representations of semidirect products of a compact quantum
  group with a finite group},
        date={2019},
     journal={arXiv preprint arXiv:1909.02359},
}

\bib{phdthesis}{thesis}{
      author={Wang, Hua},
       title={Rapid decay of bicrossed products and representation theory of
  some semidirect products},
        type={Ph.D. Thesis},
        date={2020},
}

\bib{MR1096123}{article}{
      author={Woronowicz, S.~L.},
       title={Unbounded elements affiliated with {$C^*$}-algebras and
  noncompact quantum groups},
        date={1991},
        ISSN={0010-3616},
     journal={Comm. Math. Phys.},
      volume={136},
      number={2},
       pages={399\ndash 432},
  url={http://projecteuclid.org.scd1.univ-fcomte.fr/euclid.cmp/1104202358},
      review={\MR{1096123}},
}

\bib{MR1616348}{incollection}{
      author={Woronowicz, S.~L.},
       title={Compact quantum groups},
        date={1998},
   booktitle={Sym\'{e}tries quantiques ({L}es {H}ouches, 1995)},
   publisher={North-Holland, Amsterdam},
       pages={845\ndash 884},
      review={\MR{1616348}},
}

\end{biblist}
\end{bibdiv}
\end{document}